\documentclass[12pt]{article}
\usepackage{latexsym}
\usepackage{amsthm}
\usepackage{amsmath,amssymb,amsfonts,bm,amsbsy,bbm}
\usepackage{epsfig}
\usepackage{graphicx,rotating,lscape}
\usepackage{xcolor}
\usepackage{soul}
\usepackage{hyperref}
\usepackage{verbatim}
\usepackage{natbib}
\usepackage{multirow,color}
\usepackage{geometry}
\usepackage{setspace}
\usepackage{subfigure}
\usepackage{bbm}
\usepackage[flushleft]{threeparttable}
\usepackage{enumitem}
\theoremstyle{plain}

\newtheorem{theorem}{Theorem}[section]
\newtheorem{lemma}[theorem]{Lemma}
\newtheorem{corollary}[theorem]{Corollary}
\newtheorem{proposition}[theorem]{Proposition}
\theoremstyle{definition}
\newtheorem{definition}[theorem]{Definition}

\newtheorem{assumption}{Assumption} 
\theoremstyle{remark}

\newcommand\lo[1]{_{#1}}

\def\ka{\kappa}
\newcommand\ca[1]{{\cal{#1}}}
\def\real{\mathbb{R}}

\def\red#1{{\color{red}{#1}}}

\newcommand{\indep}{\;\, \rule[0em]{.03em}{.65em} \hspace{-.45em}\rule[-.05em]{.65em}{.03em} \hspace{-.45em}\rule[0em]{.03em}{.65em}\;\,}

\def\wh{\widehat}
\def\wt{\widetilde}
\def\n{\nonumber}   
\def\var{\mbox{var}}
\def\cov{\mbox{cov}}

\def\trans{^{\top}}

\def\sumi{\sum_{i=1}^n}

\def\sumj{{\sum_{j=1}^n}}

\def\hs{\lo {\mathrm{HS}}}
\def\op{\lo {\mathrm{OP}}}
\def\reg{\lo {\mathrm{reg}}}
\def\res{\lo {\mathrm{res}}}

\def\real{\mathbb R}
\def\natural{\mathbb N}

\def\var{\mathrm{var}}

\def\sgn{\mathrm{sgn}}
\def\ran{\mathrm{ran}}
\def\ker{\mathrm{ker}}
\def\span{\mathrm{span}}
\def\cran{\overline{\mathrm{ran}}}

\def\cov{\mathrm{cov}}

\def\0{{\bf 0}}
\def\1{{\bf 1}}

\def\R{{\mathbb{R}}}

\def\b0{{\bf 0}}

\def\mF{\mathcal{F}}

\def\mH{\mathcal{H}}

\def\OP{\text{OP}}
\def\HS{\text{HS}}

\def\bse{\begin{eqnarray*}}
\def\ese{\end{eqnarray*}}
\def\be{\begin{eqnarray}}
\def\ee{\end{eqnarray}}
\def\bsq{\begin{equation*}}
\def\esq{\end{equation*}}
\def\bq{\begin{equation}}
\def\eq{\end{equation}}

\allowdisplaybreaks

\title{ On Sharpened Convergence Rate of Generalized Sliced Inverse Regression for Nonlinear Sufficient Dimension Reduction}
\author{Chak Fung Choi, Yin Tang\thanks{Corresponding author. Email address: yin.tang@uky.edu}\ \  and Bing Li}
\date{}

\begin{document}
\maketitle
\begin{abstract}
    Generalized Sliced Inverse Regression (GSIR) is one of the most important methods for nonlinear sufficient dimension reduction. As shown in \cite{li2017nonlinear}, it enjoys a convergence rate that is independent of the dimension of the predictor, thus avoiding the curse of dimensionality. In this paper we establish an improved convergence rate of GSIR under additional mild eigenvalue decay rate and smoothness conditions. Our convergence rate can be made arbitrarily close to $n^{-1/3}$ under appropriate decay rate and smoothness parameters. As a comparison, the rate of \cite{li2017nonlinear} is $n^{-1/4}$ under the best conditions. This improvement is significant  because, for example, in a semiparametric estimation problem involving an infinite-dimensional nuisance parameter, the convergence rate of the estimator of the nuisance parameter is often required to be faster than $n^{-1/4}$ to guarantee desired semiparametric properties such as asymptotic efficiency. This can be achieved by the improved convergence rate, but not by the original rate. The sharpened convergence rate can also be established for GSIR in more general settings, such as functional sufficient dimension reduction. 
\end{abstract}

 \textbf{Keywords}: sufficient dimension reduction, Generalized Sliced Inverse Regression, reproducing kernel Hilbert space, linear operator, convergence rate

\def\red#1{{\color{red}{#1}}}

\section{Introduction}
For regression problems with    high-dimensional predictors,  sufficient dimension reduction (SDR) provides  a powerful framework for finding a low-dimensional representation of the predictor that preserves all the information useful for predicting the response. The theoretical foundation of SDR builds on the concept of sufficiency, which posits that certain functions of the predictors capture all the information about the response. Consequently, the remaining predictors can be ignored without any loss of information.  SDR facilitates data visualization via low-dimensional representations of the predictors, performs data summarization without losing information, and enhances prediction accuracy by alleviating the curse of dimensionality. 

Classic linear SDR assumes the existence of a $p \times d$ matrix $B$, with $d<p$, such that $Y$ is independent of $X$ conditioning on $B \trans X$. In symbols, 
\be\label{eq:linear-sdr}
Y \indep X | B^\top X.
\ee
If this relation holds, the low-dimensional representation $B^\top X$ {  serves} as a sufficient  predictor for $Y$ since the conditional distribution of $Y$ given $X$ is fully characterized by $B^\top X$. Note that matrix $B$ in \eqref{eq:linear-sdr} is only identifiable up to an invertible right transformation. Thus, the identifiable parameter to estimate is the column space of $B$, denoted by $\span(B)$. The central space, denoted by $\ca S_{Y|X}$, is defined as the intersection of all subspaces spanned by the columns of $B$ that satisfy \eqref{eq:linear-sdr}. It is the target of estimation in linear SDR, which was first proposed and studied by \cite{Li1991}. See \cite{li2018sufficient} and \cite{ma2013review} for details. 
Many methods have been proposed to find $\ca S_{Y|X}$, 
such as sliced inverse regression (SIR, \cite{Li1991}),
sliced average variance estimate (SAVE,\cite{Cook1991}), contour regression (CR, \cite{Li2005}) and directional regression (DR, \cite{li2007}).

A closely related problem, called SDR for conditional mean, assumes the existence of a $p \times d$ matrix $B$, with $d<p$, such that
\be\label{eq:linear-sdr-mean}
E(Y|X) = E(Y|B^\top X),
\ee
which was proposed  in \cite{Cook2002} and \cite{cook-li-2004}.  Clearly, \eqref{eq:linear-sdr-mean} is a weaker condition compared to \eqref{eq:linear-sdr}, which is useful in many regression settings. The target of estimation in this problem is the central mean space, denoted by $\ca S_{E(Y|X)}$, which is the intersection of all the subspaces spanned by the columns of $B$ satisfying \eqref{eq:linear-sdr-mean}. 
Methods that target the central mean space  include, among others, ordinary least squares (OLS,  \cite{li-duan-1989}), principal Hessian directions (PHD,  \cite{li-1992}), iterative Hessian transformation (IHT,  \cite{Cook2002,cook-li-2004}),  outer product gradient (OPG, \cite{xia-tong-li-zhu-2002}) and minimum average variance estimation (MAVE, \cite{xia-tong-li-zhu-2002}).

The methodology of sufficient  dimension reduction  was extended to a nonlinear setting by several authors, where $B \trans X$ is replaced by a set of nonlinear functions. See 
 \cite{wu2008}, \cite{wang-yu-2008}, \cite{yeh-huang-lee-2009}, \cite{li-artemiou-chiaromonte-2011},   \cite{lee2013}, and \cite{li2017nonlinear}.  In the following we adopt the reporducing kernel Hilbert space (RKHS) framework articulated in \cite{li2018sufficient}. Suppose there exist functions 
  $f_1, \dots, f_d: \R^p \to \R$, with $d < p$   such that 
\be\label{eq:nonlinear-sdr}
Y \indep X | f_1(X), \dots, f_d(X).
\ee
In the above relation, the functions $f_1, \ldots, f_d$ are not identifiable, because any one-to-one transformation of $(f_1 (X), \ldots, f_d (X))$ would satisfy the same relation. The identifiable object is  the $\sigma$-field generated by $f_1(X), \dots, f_d(X)$, denoted by $\sigma\{f_1(X),\ldots, f_d(X)\}$. The goal of nonlinear SDR is to recover this  $\sigma$-field, or any set of functions generating this $\sigma$-field.  Two main classes of approaches to this  nonlinear SDR problem \eqref{eq:nonlinear-sdr} have been developed:  RKHS-based methods proposed by \cite{li-artemiou-chiaromonte-2011}, \cite{lee2013} and \cite{li2017nonlinear}, and deep learning based methods via various neural network structures, including \cite{liang2022nonlinear,sun2022kernel}, \cite{chen2024deep}, \cite{tang2025belted} and \cite{xu2025conditional}. 

Among the RKHS based methods, the most commonly used method is Generalized Sliced Inverse Regression (GSIR), which was first proposed by \cite{lee2013}. By leveraging nonlinear transformations of the predictor, GSIR is capable of achieving a better performance in dimension reduction than linear SDR methods. Consequently, it has been applied in various fields, such as graphical models \cite{li2024sufficient}, reliability analysis \cite{yin2022active}, and distributional data regression \cite{zhang2024nonlinear}. Furthermore, \cite{li2017nonlinear} extends GSIR to f-GSIR, a functional variant of GSIR, where both $X$ and $Y$ are random functions lying in Hilbert spaces instead of Euclidean spaces. 

A critically important property of GSIR is its convergence rate, as it is often used in conjunction with downstream nonparametric regression, conditional density estimation, and graphical estimation. The convergence rate of GSIR will directly affect the accuracy of downstream analysis. So far, the only published convergence rate we know of is that given in \cite{li2017nonlinear}, which is 
\be
\epsilon_n^{\beta \wedge 1} + \epsilon_n^{-1}n^{-1/2} \label{eq:epsilon n beta wedge}
\ee
where $\beta > 0$ is a constant representing the degree of smoothness between the predictor and the response, and $\epsilon_n \to 0$ is the  Tikhonov regularization sequence of constants.  

Inspired by the recent work of \citet{sang2026nonlinear}, 
which established convergence rates for nonlinear 
function-on-function regression in RKHS settings, 
we impose an additional assumption on the decay rate 
of the eigenvalues of the covariance operator of $X$. 
Under this strengthened condition, we obtain an improved 
convergence rate for GSIR given by
\be\label{eq:faster_rate}
 n^{-1/2} \epsilon_n^{(\beta\land 1)-1} 
 + \epsilon_n^{\beta\land 1} 
 + n^{-1} \epsilon_n^{-(3\alpha+1)/(2\alpha)} 
 + n^{-1/2}\epsilon_n^{-(\alpha+1)/(2\alpha)},
\ee
where $\alpha > 1$ characterizes the polynomial decay 
rate of the eigenvalues of the covariance operator of $X$.
It will be shown that the convergence rate (\ref{eq:faster_rate})  is always faster than the convergence rate (\ref{eq:epsilon n beta wedge}) in  the  ranges of $\beta$ and $\alpha$. In fact, as shown in \cite{li2017nonlinear}, under the condition $\beta \geq 1$, the optimal choice of  $\epsilon_n$ yields the rate in \eqref{eq:epsilon n beta wedge} to be $n^{-1/4}$. 
Similarly, as will be shown in this paper,  under $\beta \geq 1$ and arbitrarily large $\alpha$, the optimal choice of $\epsilon_n$ makes the rate in \eqref{eq:faster_rate} arbitrarily close to $n^{-1/3}$.
This improvement is crucially important because in many semiparametric estimation problems, the convergence rate of estimation of the nuisance parameters is required to be faster than $n^{-1/4}$ in order for the estimation of the parameter of interest to achieve the $n^{-1/2}$ rate or the semiparametric efficiency bounds. Thus, for semiparametric applications where SDR plays a part in estimating the infinite-dimensional nuisance parameter, the convergence rate of \cite{li2017nonlinear} is not enough, but the improved convergence rate will suffice. This was the original motivation for developing this faster rate.

The rest of the paper is organized as follows. Section \ref{sec:background} gives an overview of the theory of nonlinear sufficient dimension reduction, the regression operator, and two versions of the generalized sliced inverse regression methods (GSIR-I and GSIR-II)  to estimate the central $\sigma$-field. In Sections \ref{sec:rate} and \ref{sec:rate GSIR-II} we derive the improved convergence rates of GSIR-I and GSIR-II, respectively. In Section \ref{sec:fgsir} we give a brief outline of how to extend the results to the functional SDR setting. Some concluding remarks are made in Section \ref{sec:conclusion}. To save space, all proofs of the theoretical results are placed in the Appendix.

\section{Backgrounds of regression operators and GSIR} \label{sec:background}

\subsection{Mathematical background and notations} \label{sec:notations}

Let $(\Omega, \mF, P)$ be a probability space. Let $\Omega_X$ and $\Omega_Y$ be subsets of $\R^p$ and $\R^q$, and $X:\Omega \to \Omega_X$, $Y:\Omega \to \Omega_Y$ be Borel random vectors of dimension $p$ and $q$, respectively. Let $P_X$ and $P_Y$ denote the distributions of $X$ and $Y$. Let $L_2(P_X)$ denote the space of all measurable functions of $X$ having finite second moment under $P_X$. Define $L_2(P_Y)$ analogously. Let $\ka_X$ and $\ka_Y$ be  positive definite kernels on $\Omega_X \times \Omega_X$ and $\Omega_Y \times \Omega_Y$, respectively, and let $\mH_X$ and $\mH_Y$ be the corresponding reproducing kernel Hilbert spaces. For a Hilbert space $\mH$, we use $\langle \cdot, \cdot \rangle_{\ca H}$ to denote the inner product in $\mH$, and use $\|\cdot\|_\mH$ to denote the norm induced by this inner product. Furthermore, for two Hilbert spaces $\mH_1, \mH_2$,  let $\ca B(\mH_1, \mH_2)$ denote the collection of all bounded linear operators from $\mH_1$ to $\mH_2$. For a bounded linear operator $A \in \ca B(\mH_1, \mH_2)$, we use $\ker(A)$ to denote the kernel or null space  of $A$; that is, $\ker (A) = \{ x: A(x) = 0 \}$. We use $\ran(A)$ to denote the range of $A$; that is,  $\ran (A) = \{ A (x): x \in \ca H_1 \}$. Since $\ran(A)$ is a linear subspace that may not be closed, we use $\overline{\ran}(A)$ to denote the closure of $\ran(A)$. We use $A^*$ to denote the adjoint operator of $A$. For a subset $\ca V$ of a Hilbert space, we use $\span(\ca V)$ to denote the linear span of $\ca V$. We use $\overline{\span}(\ca V)$ to denote the closure of $\span(\ca V)$.

For an operator $A \in \ca B(\mH_1, \mH_2)$ that may not be invertible, we define its Moore-Penrose inverse as follows. Let $ \breve A$ denote the restriction of $A$ to $\ker(A)^\perp$. Then, $\breve A$ is surjective from  $\ker(A)^\perp$ onto $\ran(A)$. The Moore-Penrose inverse of $A$, denoted by $A^\dagger$ $:\ran(A) \to \ker(A)^\perp$, is defined by $A^\dagger y = \breve A^{-1} y$ for each $y \in \ran(A)$. In general, if $\ran(A)$ is not closed, the Moore-Penrose inverse need not be a bounded linear operator. For a comprehensive treatment of the Moore-Penrose inverse in Hilbert spaces, see \cite{hsing2015theoretical}, Section 3.5. Given two arbitrary positive sequences $a_n$ and $b_n$, we write $a_n \prec b_n$ if $a_n/b_n \to 0$, write
$a_n \preceq b_n$ if $a_n/b_n$ is bounded, and write $a_n \asymp b_n$ if $a_n \preceq b_n$ and $b_n \preceq a_n$. For two real numbers $a$ and $b$, we write  $a \wedge b$ for $\min(a,b)$.

\subsection{Regression operator} \label{sec:operator}

The construction of GSIR relies on the regression operator in RKHS. In this subsection, we introduce the concepts of the regression operator under the RKHS setting. For detailed discussions of regression operators, see, for example, \cite{lee2016}, \cite{li2018linear} and Chapter~13 of \cite{li2018sufficient}. 
We make the following assumptions about the RKHS's  $\mH_X, \mH_Y$ and the kernels $\ka_X, \ka_Y$. 

\begin{assumption}\label{ass:dense}
    $\mH_X$ and  $\mH_Y$ are dense subsets of $L_2(P_X)$ and $L_2(P_Y)$ modulo constants, that is, for any $f \in L_2(P_X)$, there is a sequence $\{f_n\}\subset \mH_X$ such that $\var[f_n(X) - f(X)] \to 0$, and a similar condition holds for $\mH_Y$.
\end{assumption}

\begin{assumption} \label{ass:k bdd}
    $\ka_X:\Omega_X \times \Omega_X\to \R$ and $\ka_Y:\Omega_Y\times \Omega_Y \to \R$ are bounded and continuous kernels.
\end{assumption}
An immediate consequence of Assumption~\ref{ass:k bdd} is  $E\{\ka_X(X,X)\} < \infty$, and $E\{\ka_Y(Y,Y)\} < \infty$, which ensures the mean elements and covariance operators are well defined in the RKHS's $\mH_X$ and $\mH_Y$. Specifically, the mean elements in $\mH_X$ and $\mH_Y$ are defined as
\bse
\mu_X = E\{\ka_X(\cdot, X)\} \in \mH_X \quad \text{and} \quad \mu_Y = E\{\ka_Y(\cdot, Y)\} \in { \mH_Y}.
\ese
The covariance operators in $\mH_X$ and $\mH_Y$ are defined as
\bse
\Sigma_{XX} &=& E [ \{ \ka_X(\cdot,X) - \mu_X\} \otimes \{\ka_X(\cdot,X) - \mu_X\} ] :\mH_X \to \mH_X, \\
\Sigma_{YY} &=& E [ \{ \ka_Y(\cdot,Y) - \mu_Y\} \otimes \{\ka_Y(\cdot,Y) - \mu_Y\} ] :\mH_Y \to \mH_Y,
\ese
the cross-covariance operator from $\mH_Y$ to $\mH_X$ is defined as
\bse
\Sigma_{XY} &=& E [ \{ \ka_X(\cdot,X) - \mu_X\} \otimes \{\ka_Y(\cdot,Y) - \mu_Y\} ] : \mH_Y \to \mH_X,
\ese
and the cross-covariance operator from $\mH_X$ to $\mH_Y$ is defined as its adjoint operator $\Sigma_{YX} = \Sigma_{XY}^*$.

Assumption~\ref{ass:k bdd} is stronger than the conditions typically imposed in the sufficient dimension reduction literature, and also stronger than those required for the covariance operator to be well defined. We impose this stronger condition primarily to facilitate the proof of the sharpened convergence rate. The boundedness of the kernels immediately implies the following embedding conditions, which are often taken as  explicit assumptions.  See, for example, \cite{lee2013}, \cite{li2017nonlinear}.  This assumption is mild and it is satisfied by commonly used kernels such as the Gaussian and Laplace kernels. 
\begin{proposition} \label{pro:var-bdd}
    Under Assumption~\ref{ass:k bdd}, there are constants $C_1 > 0$ and $C_2 > 0$ such that, for all $f \in \mH_X$, and $g \in \mH_Y$,  $\var\{f(X)\} \leq C_1 \|f\|_{\mH_X}^2$ and $\var\{g(Y)\} \leq C_2 \|g\|_{\mH_Y}^2 $. 
\end{proposition}
By Proposition~\ref{pro:var-bdd} and the Riesz representation theorem,  it follows that the mean elements  $\mu_X$ and $\mu_Y$ are the unique elements in $\mH_X$ and $\mH_Y$ such that 
\bse
\langle f, \mu_X \rangle_{\mH_X} = E\{f(X)\} \, \text{ for all $f \in \mH_X$}, \quad \text{and} \quad  \langle g, \mu_Y \rangle_{\mH_Y} = E\{g(Y)\} \, \text{ for all $g \in \mH_Y$}.
\ese
Moreover, $\Sigma_{XX}$, $\Sigma_{YY}$, $\Sigma_{XY}$, $\Sigma_{YX}$ are  the unique operators that satisfy 
\bse
&&\langle f, \Sigma_{XX} f' \rangle_{\mH_X} = \cov\{f(X), f'(X)\}, \quad 
\langle g, \Sigma_{YY} g' \rangle_{\mH_Y} = \cov\{g(Y), g'(Y)\}, \\
&&\langle f, \Sigma_{XY} g \rangle_{\mH_X} = \langle g, \Sigma_{YX} f \rangle_{\mH_Y} = \cov\{f(X), g(Y)\}, 
\ese
for all $f,f' \in \mH_X$ and $g,g' \in \mH_Y$.
To define the regression operator, we also need the following assumption.
\begin{assumption}\label{ass:range}
     $\ran(\Sigma_{XY}) \subseteq \ran(\Sigma_{XX})$.
\end{assumption}
Assumption~\ref{ass:range} is mild,  as it is slightly stronger than $\ran(\Sigma_{XY}) \subseteq \overline{\ran}(\Sigma_{XX})$, which always holds. This assumption is also proposed as part of Theorem~13.1 of \cite{li2018sufficient} and Assumption~3 of \cite{li2017nonlinear}. Under this condition, the operator
\be\label{eq:reg-op}
R_{XY} = \Sigma_{XX}^\dagger \Sigma_{XY}
\ee
is well defined, since
the domain of $\Sigma_{XX}^\dagger$ is $\ran(\Sigma_{XX})$. This operator is called the regression operator. As argued in \cite{li2018sufficient} and \cite{li2017nonlinear}, while the Moore-Penrose inverse $\Sigma_{XX}^\dagger$ is typically unbounded, it is nevertheless reasonable to assume that $\Sigma_{XX}^\dagger \Sigma_{XY}$ is bounded. This pertains to assuming a certain smoothness between the relation of $X$ and $Y$. We need the boundedness so that the regression operator can be meaningfully estimated at the sample level.

\begin{assumption} \label{ass:rbd}
    $R_{XY}$ is a bounded operator.
\end{assumption}
By definition, the adjoint operator of $R_{XY}$, $R_{XY}^* = \Sigma_{YX}\Sigma_{XX}^\dagger$, is a mapping from $\ran(\Sigma_{XX})$ to $\overline{\ran}(\Sigma_{YX})$. Under Assumption~\ref{ass:rbd}, its domain can be extended to $\overline{\ran}(\Sigma_{XX})$ by the Bounded Linear Transformation (BLT) theorem (see, for example, Theorem~1.7 of \cite{reed1980methods}). Henceforth, we use $R_{XY}^*$ to denote the extended adjoint regression operator.
As shown in the following proposition, $\cran(\Sigma_{XX})$ can be explicitly written as
\bse
\mH_X^0 = \overline{\text{span}}\{\kappa_X(\cdot, x) - \mu_X: x \in \Omega_X\}.
\ese
\begin{proposition}\label{prop:kernel}
 Under Assumption~\ref{ass:k bdd}, $\ker(\Sigma_{XX}) = (\mH_X^0)^\perp$ and $\overline{\ran}(\Sigma_{XX}) = \mH_X^0$.
\end{proposition}

The following result establishes an important connection between the conditional expectation $E\{\ka_Y(\cdot,Y) \mid X\}$ and the kernel  $\ka_X(\cdot,X)$, thereby justifying the terminology of the regression operator. 
\begin{lemma}\label{lem:ce formula}
    Under Assumptions \ref{ass:dense}--\ref{ass:rbd}, we have 
    \bse
    E\{\ka_Y(\cdot, Y)\mid X\} - \mu_Y = R_{XY}^*\{\ka_X(\cdot, X) - \mu_X\} = \Sigma_{YX}\Sigma_{XX}^\dagger \{\ka_X(\cdot, X) - \mu_X\}.
    \ese
\end{lemma}
This lemma  is similar in spirit to the result in Theorem~1 in \cite{sang2026nonlinear}.
The latter results are developed under the regression setting,   where $Y$ is not assigned a nonlinear kernel.
Moreover, in the above lemma, there is no explicit regression error that is independent of $X$, as was assumed  in  \cite{sang2026nonlinear}.

\subsection{Nonlinear SDR and GSIR} \label{sec:nonlinear-sdr}
\begin{definition}
    A sub-$\sigma$-field $\ca G$ of $\sigma(X)$ is called a sufficient dimension reduction (SDR) $\sigma$-field for $Y$ versus $X$ if 
    \be\label{eq:nonlinear-g}
    Y \indep X | \ca G.
    \ee
     If $\ca G^*$ is a sub-$\sigma$-field such that $Y \indep X | \ca G^*$, and $\ca G^* \subseteq \ca G$ for all sub-$\sigma$-fields $\ca G$ satisfying \eqref{eq:nonlinear-g}, then the sub-$\sigma$-field $\ca G^*$ is called the central dimension reduction $\sigma$-field, or the central $\sigma$-field, denoted by $\ca G_{Y|X}$.
\end{definition}
Obviously, an SDR $\sigma$-field always exists: a trivial case is $\ca G = \sigma (X)$, as   $Y \indep X | \sigma(X)$ always holds. However, this choice of $\ca G$ does not  result  in any dimension reduction. Our  goal is to find the smallest  $\sigma$-field that satisfies (\ref{eq:nonlinear-g}).   As shown in Theorem~1 of \cite{lee2013} (see also Theorem~12.2 of \cite{li2018sufficient}), under the following mild assumption such a  $\sigma$-field uniquely exists. 
\begin{assumption} \label{ass:pxy-dominated}
    The family of probability measures $\{P_{X|Y}(\cdot|y): y \in \Omega_Y\}$ is dominated by a $\sigma$-finite measure. 
\end{assumption}
Under this assumption, the intersection of all $\sigma$-fields satisfying (\ref{eq:nonlinear-g}) is itself a $\sigma$-field satisfying (\ref{eq:nonlinear-g}). This $\sigma$-field is called the central $\sigma$-field, and is denoted by $\ca G_{Y|X}$. 
Following the framework of \cite{li2025relative} and \cite{li2018sufficient}, we recast the problem of estimating an abstract central $\sigma$-field, $\ca G_{Y\mid X}$, into the estimation of  a set of functions. Using  the framework of \cite{li2017nonlinear} and \cite{li2018sufficient}, we focus on the class of functions belonging to the RKHS $\mH_X$ by making the following assumption. 

\begin{assumption} \label{ass:nonlinear-sdr}
There exist  functions $f_1,\dots,f_d \in \mH_X$ such that \eqref{eq:nonlinear-sdr} holds. Moreover the $\sigma$-field $\sigma\{f_1(X), \dots, f_d(X)\}$ is minimal. That is, for any $g_1,\dots,g_{d'} \in \mH_X$ such that \eqref{eq:nonlinear-sdr} holds, we have $\sigma\{f_1(X), \dots, f_d(X)\} \subseteq \sigma\{g_1(X), \dots, g_{d'}(X)\}$.
\end{assumption}

Assumption~\ref{ass:nonlinear-sdr} amounts to assuming   there are no redundant functions in $f_1,\dots,f_d$. Thus, $\sigma\{f_1(X), \dots, f_d(X)\}$ is indeed the central $\sigma$-field, which is our target estimand. 
In fact, the central $\sigma$-field can be fully recovered using GSIR  provided that the central $\sigma$-field is complete,  which is defined as follows.
\begin{definition}\label{def:complete}
A sub-$\sigma$-field $\ca G$ of $\sigma(X)$ is complete if, for every $\ca G$-measurable function $f$, 
$E\{f(X)|Y\} = 0$ almost surely $P_Y$ implies that $f(X) = 0$ almost surely $P_X$. 
\end{definition}

A direct application of Theorems 2 and 4 in \cite{li2025relative} gives the following theorem, which provides the theoretical foundation of GSIR and motivates an eigenvalue problem based approach for recovering the central $\sigma$-field. 
\begin{theorem}\label{thm:li25}
    Under Assumptions \ref{ass:dense}--\ref{ass:nonlinear-sdr}, we have $\sigma\{f(X): f \in \overline{\ran}(R_{XY})\} \subseteq \ca G_{Y|X}$. Furthermore, if $\ca{G}_{Y|X}$ is complete, then $\sigma\{f(X): f \in \overline{\ran}(R_{XY})\} = \ca G_{Y|X}$.
\end{theorem}

Since $\overline{\ran}(R_{XY}) = \overline{\ran}(R_{XY}AR_{XY}^*)$
for any invertible operator $A:\overline{\ran}(\Sigma_{YY}) \to \overline{\ran}(\Sigma_{YY})$, at the population level, we can plug in any invertible operator $A$ and use $\overline{\ran}(R_{XY}AR_{XY}^*)$ to recover $\ca{G}_{Y|X}$. A convenient choice of $A$ is the identity operator, in which case we use $\overline{\ran}(R_{XY}R_{XY}^*)$ to recover the central $\sigma$-field $\ca G_{Y|X}$. 
The following assumption ensures that we can make meaningful dimension reduction using Theorem~\ref{thm:li25}. 
\begin{assumption}\label{ass:rank}
    The regression operator  $R_{XY}$ defined by \eqref{eq:reg-op} has rank $d$. 
\end{assumption}

Under this assumption $\cran (R_{XY}) = \ran (R_{XY})$ and, under the assumptions in Theorem~\ref{thm:li25}, this range determines the central $\sigma$-field. Any estimation procedure that targets $\ran (R_{XY})$ is called Generalized Sliced Inverse Regression (GSIR). It turns out that $\ran (R_{XY})$ can be recovered through two different eigenvalue problems. The first, as implemented in \cite{li2017nonlinear} in the functional SDR setting, proceeds as follows. 
Let
\bse
M = \Sigma_{XX}^\dagger \Sigma_{XY} \Sigma_{YX} \Sigma_{XX}^\dagger = R_{XY} R_{XY}^*.
\ese

Note that Proposition~\ref{prop:kernel} indicates that the domain of $R_{XY}^*$ lies within $\mH_X^0$, which will be used as the feasible region upon construction of the eigenvalue problem.  Note that restricting the region within $\mH_X^0$ instead of $\mH_X$ leads to no loss of generality, because the proof of Proposition~\ref{prop:kernel} indicates that $\mH_X^0$ differs from $\mH_X$ only through an additive constant function.
However, adding a constant to $f_1,\dots,f_d$ makes no difference to the nonlinear SDR problem \eqref{eq:nonlinear-sdr}.
Based on the above discussions on the feasible region, we now introduce the eigenvalue problem to recover $\ran(R_{XY})$ as the following corollary.

\def\spn{\mathrm{span}}

\begin{corollary}\label{cor:eigenvalue2} Suppose Assumptions \ref{ass:dense}--\ref{ass:rank} hold. Let $\phi_1, \ldots, \phi_d$ be solution to the following sequential maximization problem: for each $k = 1, \ldots, d$,
\be\label{eq:eigen-m}
\max_\phi && \langle \phi, M \phi \rangle_{\mH_X}, \\
\mathrm{s.t.} && \phi \in \ca H_X^0, \quad \langle \phi, \phi \rangle_{\mH_X}=1, \quad \langle \phi, \phi_j \rangle_{\mH_X} = 0, \, j=1,\dots,k-1.\n
\ee
Then $\sigma\{\phi_1(X), \ldots, \phi_d(X)\} \subseteq \ca G _{Y|X}$. Furthermore, if $\ca G_{Y|X}$ is complete, then these functions generate the central $\sigma$-field; that is,   $ \sigma\{\phi_1(X), \ldots, \phi_d(X)\} = \ca G_{Y|X}$.
\end{corollary}

Alternatively, we can recover $\ran (R_{XY})$ by solving a slightly different eigenvalue problem; this version was implemented in \cite{li2018sufficient}. Let $\Sigma_{XX}^{\dagger 1/2}$ denote the Moore-Penrose inverse of the operator $\Sigma_{XX}^{1/2}$; that is, 
$\Sigma_{XX}^{\dagger 1/2} = (\Sigma_{XX}^{1/2})^{\dagger}$.  Define
\be\label{eq:reg-op-2}
R_{XY}' = \Sigma_{XX}^{\dagger 1/2} \Sigma_{XY},
\ee
and $
R_{XY}'^{*} = \Sigma_{YX}\Sigma_{XX}^{\dagger 1/2}
$
as the adjoint operator of $R_{XY}'$. Let
\bse
    M' = \Sigma_{XX}^{\dagger 1/2} \Sigma_{XY} \Sigma_{YX} \Sigma_{XX}^{\dagger 1/2} = R_{XY}' R_{XY}'^{*}. 
\ese

The next corollary, parallel to Corollary \ref{cor:eigenvalue2}, describes the relation between $\ran ( R_{XY}')$ and the eigenfunctions of $M'$. Before stating the corollary, we provide a proposition in parallel to Proposition~\ref{prop:kernel}, which justifies the usage of $\mH_X^0$ as the feasible region, as well as an additional assumption in parallel to Assumption~\ref{ass:rank}.

\begin{proposition}\label{prop:kernel2}
    Under Assumption~\ref{ass:k bdd}, $\ker(\Sigma_{XX}^{1/2}) = (\mH_X^0)^\perp$ and $\overline{\ran}(\Sigma_{XX}^{1/2}) = \mH_X^0$.
\end{proposition}

\begin{assumption}\label{ass:rank2}
    The operator $R_{XY}'$ defined by \eqref{eq:reg-op-2} has rank $d$. 
\end{assumption}

\begin{corollary}\label{cor:eigenvalue} Suppose Assumptions \ref{ass:dense}--\ref{ass:nonlinear-sdr}, \ref{ass:rank2} hold. 
Let $\psi_1, \ldots, \psi_d$ be solution to the following sequential maximization problem: for each $k = 1, \ldots, d$,
\be\label{eq:eigen-m1}
\max_\psi&& \langle \psi, M' \psi\rangle_{\mH_X}, \\
\mathrm{s.t.} && \psi\in \ca H_X^0, \quad \langle \psi, \psi \rangle_{\mH_X}=1, \quad \langle \psi, \psi_j \rangle_{\mH_X} = 0, \, j=1,\dots,k-1. \n
\ee
Then $\sigma\{ \Sigma_{XX}^{\dagger 1/2} \psi_1(X), \ldots, \Sigma_{XX}^{\dagger 1/2}  \psi_d(X)\} \subseteq \ca G _{Y|X}$. Furthermore, if $\ca G_{Y|X}$ is complete, then these functions generate the central $\sigma$-field; that is,   $ \sigma\{ \Sigma_{XX}^{\dagger 1/2} \psi_1(X), \ldots,  \Sigma_{XX}^{\dagger 1/2} \psi_d(X)\} = \ca G_{Y|X}$.
\end{corollary}

At the sample level, we use the empirical analogues of the eigenvalue problems in Corollaries \ref{cor:eigenvalue2} and \ref{cor:eigenvalue} to estimate the $\ran ( R_{XY})$. For easy reference, we refer to the GSIR based on Corollary \ref{cor:eigenvalue2} as GSIR-I, and that based on Corollary \ref{cor:eigenvalue} as GSIR-II.  In the next two sections, we develop rates of GSIR-I and GSIR-II  that are faster than the rate given in \cite{li2017nonlinear}, with Section \ref{sec:rate} devoted to GSIR-I and Section \ref{sec:rate GSIR-II} devoted to GIR-II.

\section{Convergence rate of GSIR-I} \label{sec:rate}

We now come to the main theme of this paper: to improve the convergence rate from \eqref{eq:epsilon n beta wedge} to \eqref{eq:faster_rate} using an additional assumption  on the decay rate of the eigenvalues of the covariance operator of $X$. 
In Section \ref{sec:operator-rate}, we analyze the convergence rate of the regression operator, while in Section \ref{sec:function-rate}, we derive the convergence rate of the eigenfunctions and express it in terms of that of the regression operator. Section \ref{sec:optimal-rate} establishes the optimal convergence rate of GSIR under different smoothness assumptions.

\subsection{Convergence rate of the estimated regression operator}\label{sec:operator-rate}

In this subsection, we first introduce our estimator for the regression operator, and then derive its convergence rate. In the following, we will use $E_n (\cdot)$ to denote the sample average: if $f$ is a function of $X$, then $E_n f(X) = n^{-1}\sum_{i=1}^n f(X_i)$. We will use $\mathbb{N}$ to denote the set of natural numbers $\{1, 2, \ldots \}$.  The estimators of the covariance and cross-covariance operators $\Sigma_{XX}$ and $\Sigma_{XY}$ are given by 
\bse
\wh{\Sigma}_{XX} &=& E_n [ \{ \ka_X(\cdot,X) - \wh{\mu}_X\} \otimes \{\ka_X(\cdot,X) - \wh{\mu}_X\} ], \\
\wh{\Sigma}_{XY} &=& E_n [ \{ \ka_X(\cdot,X) - \wh{\mu}_X\} \otimes \{\ka_Y(\cdot,Y) - \wh{\mu}_Y\} ], 
\ese
where $\wh{\mu}_X = E_n\{\ka_X(\cdot,X)\}$ and $\wh{\mu}_Y = E_n\{\ka_Y(\cdot,Y)\}$ are the empirical estimators of the mean elements. See, for example, Section 12.4 of \cite{li2018sufficient}.
We estimate the regression operator $R_{XY}$ by 
\be\label{eq:rhat}
\wh{R}_{XY} = (\wh{\Sigma}_{XX} + \epsilon_n I)^{-1} \wh{\Sigma}_{XY},
\ee
where $\epsilon_n>0$ is a Tikhonov regularization parameter. The use of Tikhonov regularization for the inverse of $\wh{\Sigma}_{XX}$ is standard {  in} nonlinear sufficient dimension reduction (\cite{lee2013}, \cite{jang2024selective}), as well as kernel ridge regression and RKHS regression (see, for example, \cite{caponnetto2007optimal}, see also Chapter~9 of \cite{steinwart2008support}).

Before turning to the convergence rate of the regression operator, we first restate a lemma concerning the convergence rates of $\wh{\Sigma}_{XX}$ and $\wh{\Sigma}_{XY}$ in terms of the Hilbert-Schmidt norm $\| \cdot \|\hs$. 
See Lemma~5 of \cite{sang2026nonlinear} or Lemma~5 of \cite{fukumizu2007statistical}. 
Since the operator norm $\| \cdot \|\op$ is no greater than the Hilbert-Schmidt norm $ \|\cdot\|\hs$, the same convergence rates also hold for the operator norm.

\begin{lemma} \label{lem:sigmaxx-hat-rate}
    Under Assumption~\ref{ass:k bdd}, $\Sigma_{XX}$ and $\Sigma_{XY}$ are Hilbert-Schmidt operators. The convergence rates of $\wh{\Sigma}_{XX}$ and $\wh{\Sigma}_{XY}$ are
    \bse
        \| \wh{\Sigma}_{XX} - \Sigma_{XX} \|\hs = O_p(n^{-1/2}), \quad 
        \| \wh{\Sigma}_{XY} - \Sigma_{XY} \|\hs = O_p(n^{-1/2}).
    \ese
\end{lemma}

Let $\{(\lambda_j,\varphi_j):j=1,2,\dots\}$ be the eigenvalue-eigenfunction sequence of $\Sigma_{XX}$ with $\lambda_1 \ge \lambda_2 \ge \dots$.
That is, we have the eigendecomposition of $\Sigma_{XX}$ as
\bse
\Sigma_{XX} = \sum_{j=1}^\infty \lambda_j (\varphi_j \otimes \varphi_j). 
\ese
Under Assumption~\ref{ass:k bdd}, $\Sigma_{XX}$ is a trace-class operator --- that is,  its eigenvalues are summable. The following assumption, which is often made in the functional data analysis literature, is the key to the sharpening of the convergence rate of GSIR. The existing convergence rate of GSIR, such as given in \cite{li2017nonlinear}, {did not} use this assumption. 
\begin{assumption}\label{ass:alpha}
$\lambda_j \asymp j^{-\alpha}$ for some $\alpha>1$ and for all $j \in \natural$.
\end{assumption}

Next,  following  the construction in \cite{sang2026nonlinear}, we  define the population-level residual 
\be\label{eq:u}
U = \ka_Y(\cdot,Y) - E\{ \ka_Y(\cdot,Y)|X\} \in \ca H_Y.
\ee 
Clearly, $\mu_U = E(U) =0$. Also, let $\Sigma_{XU} = E[\{\ka_X(\cdot,X) - \mu_X\} \otimes U]$.
We define the sample-level counterparts of $\mu_U$ and $\Sigma_{XU}$ as 
\bse
\wh{\mu}_U = E_n(U) = E_n [\ka_Y(\cdot, Y)- E\{\ka_Y(\cdot,Y)|X\}] 
= n^{-1} \sumi \{ \ka_Y(\cdot, Y_i) - E[\ka_Y(\cdot,Y_i )|X_i]\},  \\
\wh{\Sigma}_{XU} = E_n[\{\ka_X(\cdot,X) - \wh{\mu}_X\} \otimes (U - \wh{\mu}_U)] 
= n^{-1} \sumi \{\ka_X(\cdot,X_i) - \wh{\mu}_X\} \otimes (U_i - \wh{\mu}_U). \hspace{.35in}
\ese
We then define an intermediate operator between $\Sigma_{XU}$ and $\wh \Sigma_{XU}$ by replacing $\wh \mu_ X $ and $\wh \mu_U$ above by $\mu_X$ and $\mu_U = 0$: 
\bse
\wt{\Sigma}_{XU} = E_n[\{\ka_X(\cdot,X) - \mu_X\} \otimes U].
\ese
Under these definitions, one can verify that  Lemmas~6 and 7 in \cite{sang2026nonlinear} still hold. We restate them below for completeness. The proofs are omitted as they are similar to those given in \cite{sang2026nonlinear}. 

\begin{lemma}\label{lem:sigma-xu}
    Under Assumptions \ref{ass:dense}--\ref{ass:rbd}, we have
    \begin{enumerate}
        \item $\Sigma_{XU}=0$;
        \item $\wh{\Sigma}_{XY} = \wh{\Sigma}_{XU} + \wh{\Sigma}_{XX}R_{XY}$.
    \end{enumerate}
\end{lemma}

\begin{lemma}\label{lem:sigma-xu-rate}
     Under Assumption~\ref{ass:k bdd}, we have $\|\wh{\Sigma}_{XU} - \wt{\Sigma}_{XU}\|\hs = O_p(n^{-1})$.
\end{lemma}

We also need the following assumption, which pertains to a type of  smoothness in the relation between $X$ and $Y$. 

\begin{assumption}\label{ass:beta}
    There exists some $\beta>0$ such that $\Sigma_{XY} = \Sigma_{XX}^{1+\beta} S_{XY}$ for some bounded linear operator $S_{XY}:\ca H_Y \to \ca H_X$.
\end{assumption}
As discussed in \cite{li2017nonlinear}, \cite{li2018linear} and \cite{sang2026nonlinear}, Assumption~\ref{ass:beta}  requires $\Sigma_{XX}^{\dagger(1+\beta)} \Sigma_{XY}$ to be bounded for some $\beta>0$, which requires that the singular subspaces associated with the small singular values of $\Sigma_{XY}$ align closely with the eigenspaces of $\Sigma_{XX}$ corresponding to its small eigenvalues; equivalently, the leading singular directions of $\Sigma_{XY}$ lie largely within the eigenspaces associated with the larger eigenvalues of $\Sigma_{XX}$. In other words, the dominant outputs of $\Sigma_{XY}$ lie in the low-frequency region of the spectrum  of the operator $\Sigma_{XX}$, reflecting an intrinsic smoothness in the relationship between \( X \) and \( Y \). Moreover, this tendency becomes more pronounced as \( \beta \) increases. 
The following theorem gives the convergence rate of $\wh{R}_{XY} - R_{XY}$.
\begin{theorem}\label{thm:r-rate}
    Under Assumptions \ref{ass:dense}--\ref{ass:rank}, \ref{ass:alpha}--\ref{ass:beta} and $\epsilon_n \prec 1$, we have 
    \bse
    \|\wh{R}_{XY} - R_{XY}\|\op = O_p(n^{-1/2} \epsilon_n^{(\beta\land 1)-1} + \epsilon_n^{\beta\land 1} + n^{-1} \epsilon_n^{-(3\alpha+1)/(2\alpha)} + n^{-1/2}\epsilon_n^{-(\alpha+1)/(2\alpha)}).
    \ese
\end{theorem}

While Theorem~\ref{thm:r-rate} resembles Theorem~9 in \cite{sang2026nonlinear}, their derivation requires a model assumption between $Y$ and $X$ and an independence assumption between $U$ and $X$. Neither of these conditions is available for GSIR-I (or GSIR-II). Our result avoids these assumptions to adapt their proof to the present situation.
This added generality comes at the cost of imposing a slightly stronger boundedness
requirement on the kernel, as assumed in Assumption~\ref{ass:k bdd}.

\subsection{Convergence Rate of Eigenfunctions} \label{sec:function-rate}
Corollary \ref{cor:eigenvalue2} shows that the central $\sigma$-field is generated by the first $d$ eigenfunctions of $M$. By \eqref{eq:rhat}, the sample estimator of $M$ is  
\bse
\wh{M} = (\wh{\Sigma}_{XX} + \epsilon_n I)^{-1}\wh{\Sigma}_{XY} \wh{\Sigma}_{YX} (\wh{\Sigma}_{XX} + \epsilon_n I)^{-1}= \wh{R}_{XY} \wh{R}_{XY}^*. 
\ese
At the sample level, GSIR recovers the central $\sigma$-field using the $\sigma$-field generated by the first $d$ eigenfunctions of $\wh{M}$. That is, we  solve the problem \eqref{eq:eigen-m} with $M$ replaced by $\wh{M}$. 

Let $\{(\mu_j,\phi_j):j=1,2,\dots ,d\}$ denote the eigenvalue-eigenfunction sequence of $M$ with $\mu_1 \ge \mu_2 \ge \dots \ge \mu_d$, and let $\{(\wh{\mu}_j,\wh{\phi}_j):j=1,2,\dots  ,d\}$  be those of $\wh{M}$ with $\wh{\mu}_1 \ge \wh{\mu}_2 \ge \dots \ge \wh{\mu}_d$. Classical perturbation theory guarantees that the projection operators onto the eigenspaces of $\wh{M}$ converge to those of $M$ at the same rate as $\|\wh{M} - M\|\op$. See, for example, Theorem~2 in \cite{zwald2005convergence} and Lemma~1 in \cite{koltchinskii2017normal}. Moreover, the corresponding eigenfunctions converge at the same rate when their directions are aligned. The following theorem states that all of these convergence rates are governed by the convergence rate of the regression operator.

\begin{theorem}\label{thm:eigen-rate}
    Suppose that $M = R_{XY}R_{XY}^*$, $\wh{M} = \wh{R}_{XY}\wh{R}_{XY}^*$, and $\|\wh{R}_{XY} - R_{XY}\|\op = O_p(r_n)$. Then, we have $\|\wh{M} - M\|\op = O_p(r_n)$. Further suppose that $R_{XY}$ satisfies Assumption~\ref{ass:rank} and all nonzero eigenvalues of $M$ are distinct. Let $\phi_j$ and $\wh{\phi}_j$ be the $j$th eigenfunctions associated with the $j$th largest eigenvalues of $M$ and $\wh{M}$, respectively, for $j=1,\dots,d$. Let  $P_j$ and $\wh{P}_j$ be the projection operators onto the subspaces spanned by $\phi_j$ and $\wh{\phi}_j$, respectively, for $j=1, \ldots, d$. Then, we have $\|\wh{P}_{j} - P_j\|\op = O_p(r_n)$. Moreover, $\|\wh{\phi}_j - s_j \phi_j\|_{\ca H_X} = O_p(r_n)$, where $s_j = \sgn \langle \wh{\phi}_j, \phi_j \rangle_{\ca H_X}$.
\end{theorem}

Theorem~\ref{thm:eigen-rate} shows that the convergence rate of the projection operator onto each eigenspace is the same as that of the regression operator. The same holds for the convergence rate of the eigenfunctions, up to sign adjustments. However, such sign adjustments do not affect the validity of the sufficient predictors, as replacing any $\phi_j$ with $-\phi_j$ makes no difference to the relationship \eqref{eq:nonlinear-sdr}.  Combining the results in Theorems \ref{thm:r-rate} and \ref{thm:eigen-rate}, we have the convergence rate of the sufficient predictors as given in the next corollary.
\begin{corollary}\label{cor:rate for eigenfunctions}
Let $\phi_1,\dots,\phi_d$ solve \eqref{eq:eigen-m} and $\wh{\phi}_1,\dots,\wh{\phi}_d$ solve \eqref{eq:eigen-m} with $M$ replaced by $\wh{M}$, and set $s_j = \sgn \langle \wh{\phi}_j, \phi_j \rangle_{\ca H_X}$ for $j=1,\dots,d$. Under the assumptions in Theorems \ref{thm:r-rate} and \ref{thm:eigen-rate}, we have $\|\wh{\phi}_j - s_j \phi_j\|_{\ca H_X} = O_p(r_n)$, where 
\be\label{eq:rxy-rate}
r_n = n^{-1/2} \epsilon_n^{(\beta\land 1)-1} + \epsilon_n^{\beta\land 1} + n^{-1} \epsilon_n^{-(3\alpha+1)/(2\alpha)} + n^{-1/2}\epsilon_n^{-(\alpha+1)/(2\alpha)}. 
\ee
\end{corollary}

\subsection{Optimal Convergence Rate of GSIR-I} \label{sec:optimal-rate}
Note that the rate in \eqref{eq:rxy-rate} depends on the choice of tuning parameter $\epsilon_n$. In this subsection, we derive the optimal convergence rate of \eqref{eq:rxy-rate} among all possible tuning parameter rates of the form $\epsilon_n \asymp n^{-\delta}$, where $\delta > 0$ is a constant. When $\epsilon_n$ is of the form $\epsilon_n \asymp n^{-\delta}$, the convergence rate \eqref{eq:rxy-rate} becomes 
\be\label{eq:rxy-rate-delta}
 r_n \asymp n^{-1/2 + \delta\{1- (\beta \land 1)\}} + n^{-\delta(\beta \land 1)} + n^{-1 + \delta(3\alpha + 1)/(2\alpha)} + n^{-1/2 + \delta(\alpha+1)/(2\alpha)}. 
\ee
According to  Theorem~\ref{thm:r-rate}, Theorem~\ref{thm:eigen-rate}, and Corollary~\ref{cor:rate for eigenfunctions}, $r_n$ in \eqref{eq:rxy-rate-delta} is the convergence rate of the regression operator, the projection operators onto  the eigenspaces, as well as the eigenfunctions.   Let $\delta_\mathrm{opt}$ be the value of $\delta$ that minimizes \eqref{eq:rxy-rate-delta}, and let $\rho_{\mathrm{opt}}$ be the corresponding optimal rate. 
The following theorem gives the optimal choices of $\delta$ for given $\alpha$ and $\beta$, and establishes the resulting optimal convergence rate  $\rho_\mathrm{opt}$. 

\begin{theorem}\label{thm:optimal-rate}
    Suppose that all the assumptions in Theorem~$\ref{thm:r-rate}$ are satisfied. 
    \begin{itemize}
        \item if $\beta > \frac{\alpha-1}{2\alpha}$, then $\delta_{\mathrm{opt}} = \frac{\alpha}{2\alpha(\beta \land 1) + \alpha + 1}$, $\rho_{\mathrm{opt}} = n^{-\frac{\alpha(\beta \land 1)}{2\alpha(\beta \land 1) + \alpha + 1} }$.
        \item if $\beta \leq \frac{\alpha-1}{2\alpha}$, then $\delta_{\mathrm{opt}} = \frac{1}{2}$, $\rho_{\mathrm{opt}} = n^{-\frac{\beta}{2}}$.
    \end{itemize}
\end{theorem}
The proof of Theorem~\ref{thm:optimal-rate} is essentially the same as that of Theorem~10 of \cite{sang2026nonlinear} and is therefore omitted. As shown in \cite{sang2026nonlinear}, this convergence rate is always faster than the optimal rate reported in \cite{li2017nonlinear}. The reason for this improvement is we impose an additional mild Assumption~\ref{ass:alpha}, which is not made in \cite{li2017nonlinear}. When $\alpha$ is large and $\beta$ is close to 1, this rate approaches $n^{-1/3}$, which is significantly faster than the optimal rate $n^{-1/4}$ reported in \cite{li2017nonlinear}.

\section{Convergence Rate of GSIR-II}\label{sec:rate GSIR-II}

We now derive the convergence rate for  {GSIR-II}, namely for the
sample-level estimator of $\operatorname{ran}(R_{XY})$ constructed by
mimicking the population procedure described in
Corollary~\ref{cor:eigenvalue}. 
At the sample level, we estimate $R_{XY}'$ by
$ \wh R_{XY}' =
(\wh\Sigma_{XX} + \epsilon_n I)^{-1/2}\wh\Sigma_{XY}$, 
where $\epsilon_n > 0$ is a Tikhonov regularization parameter. Let $\wh \psi_1, \ldots, \wh \psi_d$ be the first $d$ eigenfunctions of $\wh M'$. We use 
\begin{eqnarray}\label{eq:eigenfunctions gsir ii}
 (\wh\Sigma_{XX} + \epsilon_n I)^{-1/2 }    \wh \psi_1, \ldots,  (\wh\Sigma_{XX} + \epsilon_n I)^{-1/2 }  \wh \psi_d
    \end{eqnarray}
to estimate $\Sigma_{XX}^{\dagger 1/2} \psi_1, \ldots, \Sigma_{XX}^{\dagger 1/2} \psi_d$, which form a basis of $\ran (R_{XY})$. Our derivation proceeds in three steps:
\begin{enumerate}
    \item establish the convergence rate of $\wh R_{XY}'$;
    \vspace{-.1in}
    \item establish the convergence rate of ${\wh M} '$;
    \vspace{-.1in}
    \item derive the convergence rate of the estimated functions in  
    \eqref{eq:eigenfunctions gsir ii}.
\end{enumerate}

\subsection{Convergence rate of $\wh R_{XY}'$}

Before presenting the main result, we first state a lemma that parallels Lemma~8 of \cite{sang2026nonlinear}.

\begin{lemma}\label{lem:alpha}
    Under Assumption~\ref{ass:alpha}, if $\epsilon_n \prec 1$, then $\sum_{j=1}^\infty \lambda_j (\lambda_j + \epsilon_n)^{-1} = O(\epsilon_n^{-1/\alpha})$.
\end{lemma}

 We now present the convergence rate of $\wh{R}_{XY}' - R_{XY}'$ in the following theorem.

\begin{theorem}\label{thm:r-rate-2}
Under Assumptions \ref{ass:dense}--\ref{ass:nonlinear-sdr}, \ref{ass:rank2}--\ref{ass:beta} and $\epsilon_n \prec 1$,  we have
\be\label{eq:r-rate-2}
\|\wh{R}_{XY}' - R_{XY}'\|\op = O_p(n^{-1/2}  \epsilon_n^{\wt{\beta} \land 1 - 1} + \epsilon_n^{\wt{\beta}\land 1} + n^{-1}\epsilon_n^{-1-1/(2\alpha)} + n^{-1/2}\epsilon_n^{-1/(2\alpha)}),
\ee
 where $\wt{\beta} = \beta + 1/2$.
\end{theorem}

\subsection{Convergence rate of the sufficient predictors}

We next derive the convergence rate of the GSIR-II  sufficient predictors that estimate a basis of $\ran (R_{XY})$. Applying Theorem~\ref{thm:eigen-rate} to 
$M' = R_{XY}'R_{XY}'^*$ and $\wh{M}' = \wh{R}_{XY}'\wh{R}_{XY}'^*$ and using the rate in Theorem \ref{thm:r-rate-2}, we derive the convergence rate of the eigenfunctions of $\wh M '$ to those of $M'$, as given by  the following corollary.

\begin{corollary}\label{cor:m-rate-2}
Suppose that all assumptions in Theorem~\ref{thm:r-rate-2} are satisfied. Then, we have $\|\wh{M}' - M' \| = O_p(r_n')$, where 
\be\label{eq:rn-2}
r_n' = n^{-1/2} \epsilon_n^{\wt{\beta} \land 1 -1} + \epsilon_n^{\wt{\beta}\land 1} + n^{-1}\epsilon_n^{-1-1/(2\alpha)} + n^{-1/2}\epsilon_n^{-1/(2\alpha)},
\ee
where $\wt{\beta} = \beta + 1/2$.
Furthermore, let $\psi_1,\dots,\psi_d$ be the first $d$  eigenfunctions of the operator $M'$ and let $\wh \psi_1, \ldots, \wh \psi_d$ be the first $d$ eigenfunctions of the operator $\wh M ' $. Let $s_j '  = \sgn \langle \wh{\psi}_j, \psi_j \rangle_{\ca H_X}$ for $j=1,\dots,d$. Further suppose that all nonzero eigenvalues of $M'$ are distinct.  Then, we have $\|\wh{\psi}_j - s_j' \psi_j\|_{\ca H_X} = O_p(r_n')$ for $j = 1, \ldots, d$.
\end{corollary}
Comparing the two rates $r_n$ and $r_n'$ in \eqref{eq:rxy-rate} and \eqref{eq:rn-2}, we observe the following points. 
First, the last two terms of \eqref{eq:rxy-rate} multiplied by $\epsilon_n^{1/2}$  become the last two terms of \eqref{eq:rn-2}. Second, replacing the $\beta$ in the first two terms of \eqref{eq:rxy-rate} by $\wt \beta = \beta + 1/2$ gives the first two terms in  \eqref{eq:rn-2}. In other words, the first two terms in \eqref{eq:rxy-rate} multiplied by $\epsilon_n^{\wt \beta \wedge 1- \beta \wedge 1}$ become  the first two terms in \eqref{eq:rn-2}.
However, note that 
\bse
\wt \beta \wedge 1 - \beta \wedge 1
=
\begin{cases}
    1/2, & 0 < \beta < 1/2, \\
    1 - \beta, & 1/2 \le  \beta < 1, \\
   0, & 1 \le \beta,  
\end{cases} \quad \Rightarrow \quad 
\wt \beta \wedge 1 - \beta \wedge 1
\begin{cases}
    > 0,  & 0 < \beta < 1, \\
   =0, & 1 \le \beta.     
\end{cases}
\ese
Thus, the rate $r_n' \prec r_n$ for $0 < \beta < 1$ and $r_n' \preceq r_n$ for $\beta \ge  1$. 

The improved rate $r_n'$ is due to  different regularization schemes in $\wh M$ and $\wh M'$: the former involves $(\wh{\Sigma}_{XX} + \epsilon_n I)^{-1}$ but the latter involves $(\wh{\Sigma}_{XX} + \epsilon_n I)^{-1/2}$. However, when we transform the eigenfunctions of $\wh M '$ to the sufficient predictors in GSIR-II, we need to multiply them  by an additional factor $(\wh{\Sigma}_{XX} + \epsilon_n I)^{-1/2}$. As a result,  the apparent gain in convergence  rate is canceled out, leading to the same convergence rate for GSIR-I and GSIR-II. This is to be expected, because the two approaches estimate the same subspace $\ran(R_{XY})$ at the population level. At the sample level, the sufficient predictor  estimators ultimately involve the same amount of regularization: GSIR-I applies $(\wh \Sigma_{XX} + \epsilon_n I)^{-1}$ once, while GSIR-II applies  $(\wh \Sigma_{XX} + \epsilon_n I)^{-1/2}$ twice. This equivalence in rates is shown in the next corollary.

\begin{corollary}\label{cor:f-rate-2}
Suppose that all conditions in Theorems \ref{thm:r-rate} and \ref{thm:r-rate-2} and Corollary \ref{cor:m-rate-2} are satisfied. Let $\wh \eta_j = (\wh{\Sigma}_{XX}+\epsilon_nI)^{-1/2} \wh{\psi}_j$ and $\eta_j = \Sigma_{XX}^{\dagger 1/2}\psi_j$, for $j=1,\dots,d$. Then, we have
\bse
\|\wh \eta_j - s_j' \eta_j \|_{\mH_X}= O_p(r_n), \quad j = 1, \ldots, d, 
\ese
where $r_n$ is defined by \eqref{eq:rxy-rate}.
\end{corollary}

Corollary \ref{cor:f-rate-2} shows that the convergence rate of the sufficient predictors for GSIR-II is the same as that of GSIR-I under similar sets of assumptions. In other words, they have the same degree of improvement over the convergence rate of GSIR stated in \cite{li2017nonlinear}, thanks to the added condition on the decay rate of the eigenvalues of the covariance operator $\Sigma_{XX}$, as postulated in Assumption~\ref{ass:alpha}.

 \section{Extension to the functional SDR setting}\label{sec:fgsir}

The results in Sections \ref{sec:rate} and \ref{sec:rate GSIR-II} can be directly extended to the functional nonlinear SDR setting \citep{li2017nonlinear}, where $X$ and $Y$ are assumed to take values in   Hilbert spaces $\ca X$ and $\ca Y$. We can adapt the nested Hilbert spaces approach in \cite{li2017nonlinear}. Specifically, we construct a second-level RKHS $\mH_X$ on $\ca X$ by imposing a positive definite kernel $\ka_X$ on $\ca X \times \ca X$, where for $x_1,x_2 \in \ca X$, $\ka_X(x_1,x_2)$ is a function of $\langle x_1, x_1 \rangle_{\ca X}$, $\langle x_2, x_2 \rangle_{\ca X}$ and $\langle x_1, x_2 \rangle_{\ca X}$. The same applies to $\mH_Y$ and $\ka_Y$. Based on these definitions, all assumptions can be similarly imposed, and all results can be applied directly to f-GSIR. Therefore, under similar assumptions, the convergence rate developed in this paper also applies to f-GSIR.

\section{Concluding Discussions}\label{sec:conclusion}

In this paper, we show that the convergence rate of GSIR  and f-GSIR can be sharpened to be arbitrarily close to $n^{-1/3}$, improving upon the best rate $n^{-1/4}$ reported in \cite{li2017nonlinear}. This refinement is obtained by imposing a mild eigenvalue decay rate assumption on the covariance operator $\Sigma_{XX}$, which improves the convergence rate of the regression operator. 

A notable feature of this convergence rate is that it is  entirely independent of the dimension $p$ of $X$, and the dimension $d$ of the sufficient predictor, as long as these dimensions do not depend on the sample size $n$. The rate depends only on the degree of smoothness between $X$ and $Y$ and on the decay rate of the eigenvalues of the covariance operator of $X$. This feature is fundamentally important in alleviating the curse of dimensionality. Specifically, for many nonparametric regression and machine learning methods such as kernel regression and kernel conditional density estimation, the convergence rate deteriorates quickly with the dimension $p$ of the predictor $X$. However, as argued in \cite{li2017nonlinear} and \cite{li2018sufficient},  it is often reasonable to assume   an underlying low-dimensional nonlinear structure in $X$---in fact, often as low as $d=1$ or 2---such that   $Y$ depends on $X$ only through the $d$ sufficient predictors. 
Since the convergence rate of dimension reduction is not affected by the original dimension $p$, if we first perform nonlinear dimension reduction on $X$ through GSIR and then feed the sufficient predictors to the downstream analysis, then the final convergence rate is determined not by the original dimension $p$, but by the reduced dimension $d$.  This is the mechanism by which we avoid the curse of dimensionality through a nonlinear sufficient dimension reduction method such as GSIR.

\def\cip{\stackrel{p}\to}
Since the focus of this paper is on establishing the convergence rate of GSIR, we have omitted some issues secondary to this theme. In particular, we discuss three issues worth mentioning to conclude this paper. The first issue is determination of the dimension of the sufficient predictor. Although this dimension $d$ is treated as given in our analysis, it is unknown in practice and must be estimated. Related methods on this direction can be found in \cite{li2017nonlinear} and \cite{li2018sufficient}. In particular, if we have a consistent order determination method with $\wh d \cip d$, we expect the improved and the original convergence rates to remain the same.  The second issue is the numerical procedures involved in solving the eigenvalue  problems stated in Corollary \ref{cor:eigenvalue2}  and Corollary \ref{cor:eigenvalue}. A standard approach is to use a coordinate representation of the operators involved, see \cite{lee2013}, \cite{li2017nonlinear} and \cite{li2018sufficient} for further details. The third issue is that the convergence rate $\epsilon_n^{\beta \wedge 1} + \epsilon_n^{-1}n^{-1/2}$ for GSIR-II without the eigenvalue decay rate assumption, Assumption~\ref{ass:alpha},  has not been officially recorded in the literature, either in the multivariate setting or in the functional setting. However,  this is very obvious from the proof \cite{li2017nonlinear} and the continuous functional calculus argument we used in Section \ref{sec:rate GSIR-II}. Thus, taken together, the two types of convergence rates for  the two estimators, GSIR-I and GSIR-II, provide a rather complete picture  of the convergence behavior of GSIR under the smoothness condition and/or the eigenvalue decay condition.

\appendix

\section*{Appendix}

\section{Proofs of results in Section \ref{sec:background}}

\begin{proof}[Proof of Proposition~\ref{pro:var-bdd}]
    By the reproducing property, for any $x \in \Omega_X$, we have 
    \bse
    |f(x)| = |\langle f, \ka_X(\cdot, x)\rangle_{\mH_X}|\leq \|f\|_{\mH_X} \|\ka_X(\cdot, x)\|_{\mH_X} = \|f\|_{\mH_X} \sqrt{\ka_X(x,x)}.
    \ese
    Hence, 
    \bse
    \var\{f(X)\} \leq E\{f(X)^2\} \leq \|f\|_{\mH_X}^2 \sup_x \ka_X(x,x).
    \ese
     Taking $C_1 = \sup_x \ka_X(x,x)$  gives the desired result for $\var\{f(X)\}$. The result for $\var\{g(Y)\}$ can be proved similarly. 
\end{proof}

\begin{proof}[Proof of Proposition~\ref{prop:kernel}]
It suffices to show that $\ker(\Sigma_{XX}) = (\mH_X^0)^\perp$. Let $f \in \ker(\Sigma_{XX})$. Then, we have $\Sigma_{XX} f = 0$ and $\var\{f(X)\} = \langle f, \Sigma_{XX} f \rangle_{\mH_X} = 0$. This implies $f(X)$ is constant almost surely $P_X$,  which further implies that $f(x) = E\{f(X)\}$ almost surely. Moreover, we have $\langle f, \ka_X(\cdot,x) - \mu_X \rangle_{\mH_X} = f(x) - E\{f(X)\} = 0$ almost surely, which gives $f \in (\mH_X^0)^\perp$. Thus, $\ker(\Sigma_{XX}) \subseteq (\mH_X^0)^\perp$. Conversely, each step above is reversible, so $\ker(\Sigma_{XX}) \supseteq (\mH_X^0)^\perp$ also holds. Summarizing the two directions gives the desired result.
\end{proof}

 \begin{proof}[Proof of Lemma~\ref{lem:ce formula}]
     By the reproducing property, for any $g \in \mH_Y$, we have
     \bse
     \langle g, E\{\ka(\cdot, Y)\mid X\} - \mu_Y\rangle_{\mH_Y} = E\{g(Y)|X\} - E\{g(Y)\}.
     \ese
     On the other hand, by Proposition~1 of \cite{li2017nonlinear}, we have
     \be\label{eq:rxyg}
     R_{XY} g = E\{g(Y)|X\} - E\{g(Y)\} + E\{R_{XY}g (X)\}.
     \ee
     Thus,
     \bse
     && \langle g, R_{XY}^*\{\ka_X(\cdot, X) - \mu_X\} \rangle_{\mH_Y} \\
     &=& \langle R_{XY} g ,  \ka_X(\cdot, X) - \mu_X \rangle_{\mH_X} \\
     &=& \langle E\{g(Y)|X\} - E\{g(Y)\} + E\{R_{XY}g (X)\}, \ka_X(\cdot, X) - \mu_X \rangle_{\mH_X} \\
     &=& E\{g(Y)|X\} - E\{g(Y)\} \\
     &=& \langle g, E \{ \ka_X (\cdot, X) | Y \} - \mu_X \rangle_{\ca H_Y},  
     \ese
     where the second equality follows from \eqref{eq:rxyg}, and the third equality follows from the reproducing property. 
    Since the inner products on the left- and right-hand side of the above equality  coincide for all $g \in \mH_Y$, we have  $R_{XY}^*\{\ka_X(\cdot, X) - \mu_X\} = E\{\ka_Y(\cdot, Y)|X\} - \mu_Y$.
 \end{proof}

\begin{proof}[Proof of Corollary \ref{cor:eigenvalue2}]
By Theorem \ref{thm:li25}, we only need to show that $\spn\{\phi_1,\dots,\phi_d\} = \ran(M)$. 
Since $M$ has a finite rank of $d$ by Assumption~\ref{ass:rank}, the above process gives the eigenfunctions of $M$ corresponding to its nonzero eigenvalues $\lambda_1 \ge \cdots \lambda_d > 0$. By spectral decomposition, 
 \bse
 M = \sum_{j=1}^d \lambda_j (\phi_j \otimes \phi_j).  
 \ese
If $f \in \ran (M)$, then $f = M g$ for some $g \in \ca H_X^0$. Then, by the above identity we have $f = \sum_{j=1}^d \lambda_j \langle \phi_i, g \rangle_{\ca H_X} \phi_i$, which is a member of $\spn \{ \phi_1, \ldots, \phi_d \}$. On the other hand, if $f \in \spn \{\phi_1, \ldots, \phi_d\}$, then, for some $c_1, \ldots, c_d \in \real$, 
\bse
    f = c_1 \phi _1 + \cdots + c_d \phi_d = M [  (c_1/ \lambda_1)  \phi _1 + \cdots + (c_d/ \lambda_d )   \phi _d  ], 
\ese
which is a member of $\ran(M)$. 
 \end{proof} 

\begin{proof}[Proof of Proposition \ref{prop:kernel2}]
    Based on Proposition \ref{prop:kernel}, we only need to show that $\ker(\Sigma_{XX}) = \ker(\Sigma_{XX}^{1/2})$. If $f \in \ker(\Sigma_{XX}^{1/2})$, then $\Sigma_{XX}f = \Sigma_{XX}^{1/2} \Sigma_{XX}^{1/2} f = 0$, which implies that $f \in \ker(\Sigma_{XX})$. Thus, we have $\ker(\Sigma_{XX}^{1/2}) \subseteq \ker(\Sigma_{XX})$. Conversely, if $f \in \ker(\Sigma_{XX})$, then $\|\Sigma_{XX}^{1/2} f \|_{\mH_X}^2 = \langle f, \Sigma_{XX} f \rangle_{\mH_X}=0$, which further implies that $f \in \ker (\Sigma_{XX}^{1/2})$. Thus, we have $\ker(\Sigma_{XX}) \subseteq \ker(\Sigma_{XX}^{1/2})$. Summarizing the above results gives $\ker(\Sigma_{XX})=\ker(\Sigma_{XX}^{1/2})$.
\end{proof}

\begin{proof}[Proof of Corollary \ref{cor:eigenvalue}]
    By the same argument used in the proof of Corollary \ref{cor:eigenvalue2}, we can show that $\ran (M') = \spn ( \psi_1, \ldots, \psi_d)$. Then 
    \bse
        \ran ( M ) = \ran ( \Sigma_{XX}^{\dagger 1/2} M'  \Sigma_{XX}^{\dagger 1/2} ) 
        = \Sigma_{XX}^{\dagger 1/2} \ran ( M'   ) 
       & = &\Sigma_{XX}^{\dagger 1/2} \spn ( \psi_1, \ldots, \psi_d )  \\
    & = &\spn ( \Sigma_{XX}^{\dagger 1/2} \psi_1, \ldots, \Sigma_{XX}^{\dagger 1/2} \psi_d ), 
    \ese
as desired. 
\end{proof}

\section{Proofs of results in Section \ref{sec:rate}}

\begin{proof}[Proof of Theorem~\ref{thm:r-rate}]
    The proof of Theorem~\ref{thm:r-rate} largely follows the argument in the proof of Theorem~9 in \cite{sang2026nonlinear}. However, their analysis requires that $U$ and $X$ are independent, which is not assumed here. Therefore, all arguments that do not involve $U$ remain valid here, while the terms involving $U$ require a different analysis.
    To simplify notation, we abbreviate $(\wh{\Sigma}_{XX} + \epsilon_n I)^{-1}$, $(\Sigma_{XX} + \epsilon_n I)^{-1}$, and $\Sigma_{XX}^\dagger$ by $\wh{V}$, $V_n$ and $V$, respectively. 
By Lemma~\ref{lem:sigma-xu}, we  can decompose $\wh{R}_{XY} $  into $\wh{R}_{\text{reg}} + \wh{R}_{\text{res}}$, where 
\bse
        \wh{R}_{\text{reg}} = \wh{V} \wh{\Sigma}_{XX} R_{XY}, \quad \wh{R}_{\text{res}} = \wh{V} \wh{\Sigma}_{XU}.
\ese
Let $R_n = V_n \Sigma_{XY}$. Further decompose $\wh{R}_{\text{reg}}$ into $\wh{R}_{\text{reg}} - R_n + R_n$, and we have 
\be \label{eq:rxy-diff}
\wh{R}_{XY} - R_{XY} = \wh{R}_{\text{res}} + (\wh{R}_{\text{reg}}  - R_n)  + (R_n - R_{XY}). 
\ee
Since the second and third terms do not involve $U$, we can follow the  argument in the proof of Theorem~9 of \cite{sang2026nonlinear} to obtain
\be\label{eq:rhat-reg-n}
\|\wh{R}_{\text{reg}} - R_n\|\op = O_p(n^{-1/2}\epsilon_n^{\beta \land 1 - 1}) \quad \text{ and } \quad  \|R_n - R_{XY}\|\op = O(\epsilon^{\beta \land 1}).
\ee
It remains to analyze the term $\wh{R}_{\text{res}}$, which can be further decomposed into
\be\label{eq:rhat-res}
\wh{R}_{\text{res}} =  (\wh{V} \wh{\Sigma}_{XU} - \wh{V} \wt{\Sigma}_{XU}) + (\wh{V} \wt{\Sigma}_{XU} - V_n \wt{\Sigma}_{XU} ) + V_n \wt{\Sigma}_{XU}.
\ee
For the first term in \eqref{eq:rhat-res}, we apply the same argument as in the proof of Theorem~9 of \cite{sang2026nonlinear}, which yields
\be\label{eq:res-1}
\|\wh{V} \wh{\Sigma}_{XU} - \wh{V} \wt{\Sigma}_{XU}\|\op = O_p(n^{-1}\epsilon_n^{-1}).
\ee
The arguments for the convergence rate of the second and third terms in \eqref{eq:rhat-res} are different from \cite{sang2026nonlinear}, as  $U$ and $X$ are not assumed independent here. Since $\wh{V} - V_n = \wh{V}(\Sigma_{XX} - \wh{\Sigma}_{XX}) V_n$, the operator norm of the second term on the right-hand side of \eqref{eq:rhat-res} is bounded by 
\be
\|\wh{V} \wt{\Sigma}_{XU} - V_n \wt{\Sigma}_{XU}\|\op &\leq& \|\wh{V}\|\op\|\Sigma_{XX} - \wh{\Sigma}_{XX}\|\op \|V_n \wt{\Sigma}_{XU}\|\op \n \\
&=& O_p(n^{-1/2} \epsilon_n^{-1}) \|V_n \wt{\Sigma}_{XU}\|\op. \label{eq:vsxu-vnsxu1}
\ee
Therefore, to derive the convergence rates of the second and third terms in \eqref{eq:rhat-res}, it remains to find the convergence rate of $\|V_n \wt{\Sigma}_{XU}\|\op$, which is bounded by $\|V_n \wt{\Sigma}_{XU}\|_\HS$. By construction,
\be\label{eq:vn-sigma-xu-hs2}
\|V_n \wt{\Sigma}_{XU}\|_\HS^2 = \left\|n^{-1}\sumi V_n[\{\ka(\cdot, X_i ) - \mu_X\} \otimes U_i]\right\|_\HS^2.
\ee
Since, by Lemma~\ref{lem:sigma-xu}, $\Sigma_{XU} = 0,$ we have
\be\label{eq:e-vn-sigma-xu-hs2}
E(\|V_n \wt{\Sigma}_{XU}\|_\HS^2) &=& n^{-2}\sumi \sumj E(\langle V_n[\{\ka(\cdot, X_i) - \mu_X\} \otimes U_i], V_n[\{\ka(\cdot, X_j) -\mu_X\} \otimes U_j]\rangle_\HS) \n\\
&=& n^{-2} \sumi E (\|V_n[\{\ka(\cdot, X_i) - \mu_X\} \otimes U_i]\|_\HS^2) \n\\
&=& n^{-1}E (\|V_n[\{\ka(\cdot, X) - \mu_X\} \otimes U]\|_\HS^2) .
\ee
Note that Lemma~4.33 of \cite{steinwart2008support} implies that $\mH_X$ and $\mH_Y$ are separable Hilbert spaces under Assumption~\ref{ass:k bdd}. 

Note that the Karhunen–Lo\'eve expansion of $\ka_X(\cdot,X)$ can be written as
\be\label{eq:kl-expansion}
\ka_X(\cdot, X) = \mu_X + \sum_{j=1}^\infty \zeta_j\varphi_j, \quad \text{where} \quad \zeta_j = \langle  \ka_X(\cdot,X) - \mu_X,  \varphi_j \rangle_{\mH_X},
\ee
 where $\zeta_j$'s are uncorrelated random variables with mean zero and $\var(\zeta_j) = \lambda_j$, for $j=1,2,\dots$. See Theorem~11.4.1 of \cite{kokoszka2017introduction} and Theorem~7.2.7 of \cite{hsing2015theoretical} for details. 

Let $\{\phi_j:j = 1,2,\ldots\}$ be an orthogonal basis of $\mH_Y$. 
The squared Hilbert-Schmidt norm on the right-hand side is 
\be\label{eq:vku-hs2}
&&\|V_n[\{\ka(\cdot, X) - \mu_X\} \otimes U]\|_\HS^2 \n\\
&=& \sum_{j=1}^\infty \langle V_n[\{\ka(\cdot, X) - \mu_X\} \otimes U] \phi_j,  V_n[\{\ka(\cdot, X) - \mu_X\} \otimes U] \phi_j \rangle_{\mH_X} \n\\ 
&=& \sum_{j=1}^\infty \langle V_n \{\ka(\cdot, X) - \mu_X\} \langle U, \phi_j \rangle_{\mH_Y}, V_n \{\ka(\cdot, X) - \mu_X\} \langle U, \phi_j \rangle_{\mH_Y}  \rangle_{\mH_X} \n\\
&=& \langle V_n \{\ka(\cdot, X) - \mu_X\}, V_n \{\ka(\cdot, X) - \mu_X\} \rangle_{\mH_X} \sum_{j=1}^\infty \langle U, \phi_j \rangle_{\mH_Y}^2 \n\\
&=& \|U\|^2_{\mH_Y} \sum_{j=1}^\infty (\lambda_j + \epsilon_n)^{-2}\zeta_j^2,
\ee
where the last equality follows from Parseval’s identity and \eqref{eq:kl-expansion}. 
Under Assumption~\ref{ass:k bdd}, we have
\be\label{eq:u:bd}
\|U\|^2_{\mH_Y} 
&=& \|\ka_{Y}(\cdot, Y) - E\{\ka_Y(\cdot, Y)|X\}\|^2_{\mH_Y} \n\\
&\le& 2\|\ka_{Y}(\cdot, Y)\|^2_{\mH_Y} + 2\|E\{\ka_Y(\cdot, Y)|X\}\|^2_{\mH_Y} \n\\
&\le& 2\|\ka_{Y}(\cdot, Y)\|^2_{\mH_Y} + 2E\{\|\ka_Y(\cdot, Y)\|^2_{\mH_Y}|X\} \n\\
&\le& 2\ka_{Y}(Y, Y) + 2E\{\ka_Y(Y, Y)|X\} \n\\
&\le& 4 C,
\ee
where $C$ is the bound of $\ka_Y$ under Assumption~\ref{ass:k bdd}.
Using this relation we deduce 
\be\label{eq:e-vku-hs2}
E(\|V_n[\{\ka(\cdot, X) - \mu_X\} \otimes U]\|_\HS^2) 
&=& E\left\{\|U\|^2_{\mH_Y} \sum_{j=1}^\infty (\lambda_j + \epsilon_n)^{-2}\zeta_j^2\right\} \n\\
&\le& 4 C E\left\{ \sum_{j=1}^\infty (\lambda_j + \epsilon_n)^{-2}\zeta_j^2\right\} \n\\
&=& 4C \sum_{j=1}^\infty \frac{\lambda_j}{(\lambda_j + \epsilon_n)^2}\n\\
&=& O(\epsilon_n^{-(\alpha + 1)/\alpha}),
\ee
 where the last line follows from Lemma~8 of \cite{sang2026nonlinear}.
Hence, $E\|V_n \wt{\Sigma}_{XU}\|_{HS}^2$ is of the order $O(n^{-1}\epsilon_n^{-(\alpha+1)/\alpha})$. By   Chebyshev's inequality, we have 
\be\label{eq:res-3}
 \|V_n \wt{\Sigma}_{XU}\|_{\OP} \le \|V_n \wt{\Sigma}_{XU}\|_{\HS} = O_p(n^{-1/2}\epsilon_n^{-(\alpha+1)/(2\alpha)}) . 
\ee
Combining this with \eqref{eq:vsxu-vnsxu1}, we have 
\be\label{eq:res-2}
 \|\wh{V} \wt{\Sigma}_{XU} - V_n \wt{\Sigma}_{XU}\|_{\OP} = O_p(n^{-1/2}\epsilon_n^{-1} n^{-1/2}\epsilon_n^{-(\alpha+1)/(2\alpha)}) = O_p(n^{-1}\epsilon_n^{-(3\alpha+1)/(2\alpha)}) .
\ee
Summarizing  \eqref{eq:rhat-res} \eqref{eq:res-1}, \eqref{eq:res-3} and \eqref{eq:res-2}, we have
\be\label{eq:rres}
\|\wh{R}_{\text{res}}\|\op &=& O_p(n^{-1}\epsilon_n^{-1} + n^{-1}\epsilon_n^{-(3\alpha+1)/(2\alpha)} + n^{-1/2}\epsilon_n^{-(\alpha+1)/(2\alpha)}) \n\\
&=& O_p(n^{-1}\epsilon_n^{-(3\alpha+1)/(2\alpha)} + n^{-1/2}\epsilon_n^{-(\alpha+1)/(2\alpha)}).
\ee
 Combining  \eqref{eq:rxy-diff}, \eqref{eq:rhat-reg-n} and \eqref{eq:rres}, we obtain the desired result.
\end{proof}

\begin{proof}[Proof of Theorem~\ref{thm:eigen-rate}]
Since 
\bse
&&\wh{R}_{XY} \wh{R}_{XY}^* - R_{XY} R_{XY}^* \\
&=&   (\wh{R}_{XY} - R_{XY}) (\wh{R}_{XY} - R_{XY})^* 
+  R_{XY} (\wh{R}_{XY} - R_{XY})^* 
 +  (\wh{R}_{XY} - R_{XY}) R_{XY} ^*, 
\ese
we have 
\bse
\|\wh{M} - M\|\op 
&=& \|\wh{R}_{XY} \wh{R}_{XY}^* - R_{XY} R_{XY}^*\|\op \\
&\le& \| (\wh{R}_{XY} - R_{XY}) (\wh{R}_{XY} - R_{XY})^*\|\op 
+ \| R_{XY} (\wh{R}_{XY} - R_{XY})^*\|\op \\
&& + \| (\wh{R}_{XY} - R_{XY}) R_{XY} ^*\|\op \\
&\le& \|\wh{R}_{XY} - R_{XY}\|\op^2 + 2 \|R_{XY}\|\op \|\wh{R}_{XY} - R_{XY}\|\op \\
&=& O_p(r_n).
\ese
By Lemma~1 in \cite{koltchinskii2017normal}, 
\be\label{eq:op-bd}
\| \wh{P}_j - P_j \|\op \le 4 \| \wh{M} - M \|\op / \delta_j,
\ee
where $\delta_j = \min (\mu_{j-1}-\mu_j, \mu_j-\mu_{j+1})$ for $j=2,\dots,d$, and $\delta_1 = \mu_1-\mu_2$. See also \cite{koltchinskii2016asymptotics} and \cite{kato1995perturbation} for details. Since $d$ is fixed, $\min \{ \delta_1, \ldots, \delta_d \}$ is a positive constant. Thus  we have proved $\|\wh{P}_{j} - P_j\|\op = O_p(r_n)$.

Next, we prove $\|\wh{\phi}_j - s_j \phi_j\|_{\ca H_X} = O_p(r_n)$. Since   $P_j = \phi_j \otimes \phi_j$ and $\wh{P}_j = \wh{\phi}_j \otimes \wh{\phi}_j$,   $P_j$ and $\wh{P}_j$ are  rank 1 operators. Therefore, the rank of $\wh{P}_j - P_j$ is at most 2. Let $\gamma_1$ and $\gamma_2$ be the first two eigenvalues of $\wh{P}_j - P_j$ with $|\gamma_1| \ge |\gamma_2|$. Then, using \eqref{eq:op-bd}, we have
\be\label{eq:hs-bd}
\| \wh{P}_j - P_j \|\hs = \sqrt{\gamma_1^2 + \gamma_2^2} \le \sqrt{2 \gamma_1^2} = \sqrt{2} |\gamma_1| = \sqrt{2} \| \wh{P}_j - P_j \|\op.  
\ee
On the other hand, we have identities
\bse
\|\wh{\phi}_j - s_j \phi_j\|_{\ca H_X}^2 = \langle \wh{\phi}_j - s_j \phi_j, \wh{\phi}_j - s_j \phi_j \rangle_{\ca H_X} = 2 ( 1 - \langle \wh{\phi}_j, s_j \phi_j \rangle_{\ca H_X}) = 2 ( 1 - | \langle \wh{\phi}_j, \phi_j \rangle_{\ca H_X}| ),
\ese
and
\bse
\| \wh{P}_j - P_j \|\hs^2 = \|\wh{\phi}_j \otimes \wh{\phi}_j - \phi_j \otimes \phi_j \|\hs^2 = 2 ( 1- \langle \wh{\phi}_j \otimes \wh{\phi}_j, \phi_j \otimes \phi_j \rangle\hs ) = 2 ( 1 - \langle \wh{\phi}_j, \phi_j \rangle_{\ca H_X}^2 ).
\ese
Since $\|\wh{\phi}_j\| = \|\phi_j\| =1$, we have $| \langle \wh{\phi}_j, \phi_j \rangle_{\ca H_X}| \le 1$, and consequently,
\be\label{eq:ef-bd}
\|\wh{\phi}_j - s_j \phi_j\|_{\ca H_X}^2 \le \| \wh{P}_j - P_j \|\hs^2.
\ee
Combining \eqref{eq:op-bd}, \eqref{eq:hs-bd} and \eqref{eq:ef-bd}, we have
\bse
\|\wh{\phi}_j - s_j \phi_j\|_{\ca H_X} \le 4\sqrt{2} \| \wh{M} - M \|\op / \delta_j = O_p (r_n),
\ese
as desired. 
\end{proof}

\section{Proofs of results in Section \ref{sec:rate GSIR-II}}

\begin{proof}[Proof of Lemma~\ref{lem:alpha}]
Let $m_n = \lfloor \epsilon_n^{-1/\alpha} \rfloor$. Then, by Assumption~\ref{ass:alpha},
\bse
\sum_{j=1}^\infty \frac{\lambda_j}{\lambda_j + \epsilon_n}
\le \sum_{j=1}^{m_n} 1 + \epsilon_n^{-1} \sum_{j=m_n+1}^\infty \lambda_j \asymp m_n + \epsilon_n^{-1} \int_{m_n}^\infty x^{-\alpha}dx \asymp \epsilon_n^{-1/\alpha}.
\ese
\end{proof}

\begin{proof}[Proof of Theorem~\ref{thm:r-rate-2}]
Define $\wh{Q}(t) = (\wh{\Sigma}_{XX}+t I)^{-1/2}$, $Q(t) = (\Sigma_{XX}+tI)^{-1/2}$, and let
\bse
\wh{R}\reg' = \wh{Q}(\epsilon_n)\wh{\Sigma}_{XX}R_{XY}, \quad \wh{R}\res' = \wh{Q}(\epsilon_n)\wh{\Sigma}_{XU}, \quad R_n'=Q(\epsilon_n)\Sigma_{XY}.
\ese 
By Lemma~\ref{lem:sigma-xu}, we can decompose $\wh R'_{XY}$ into $ \wh R'_{\reg} + \wh R'_{\res}$, which gives the following decomposition
\be
\wh{R}_{XY}' - R_{XY}' = \wh{R}\res' + (\wh{R}\reg' - R_n') + (R_n' - R_{XY}'). \label{eq:R-decomposition}
\ee
We now derive the convergence rates of the three terms on the right-hand side separately. 

\paragraph{Convergence rate for  $\wh{R}\reg' - R_n'$.} 
By construction,
\begin{eqnarray}\label{eq:hat R reg '}
\wh{R}\reg' - R_n' = \wh{Q}(\epsilon_n)\wh{\Sigma}_{XX}R_{XY} - Q(\epsilon_n) \Sigma_{XX} R_{XY} = [\wh{Q}(\epsilon_n)\wh{\Sigma}_{XX} - Q(\epsilon_n) \Sigma_{XX}] R_{XY}.
\end{eqnarray}
By equation~(3.43) in Chapter~V of \cite{kato1995perturbation},  we have
\be
Q(\epsilon_n) = (\Sigma_{XX}+\epsilon_n I)^{-1/2}
&=&\frac{1}{\pi}\int_0^\infty t^{-1/2}(\Sigma_{XX}+\epsilon_nI + tI)^{-1} dt \nonumber \\
&=& \frac{1}{\pi}\int_0^\infty t^{-1/2}Q^2(\epsilon_n + t) dt,\label{eq:q-eps}\\
\wh{Q}(\epsilon_n) = (\wh{\Sigma}_{XX}+\epsilon_n I)^{-1/2}
&=&\frac{1}{\pi}\int_0^\infty t^{-1/2}(\wh{\Sigma}_{XX}+\epsilon_nI + tI)^{-1} dt\nonumber\\
&=&\frac{1}{\pi}\int_0^\infty t^{-1/2}\wh{Q}^2(\epsilon_n + t) dt.\label{eq:qhat-eps}
\ee
Hence,
\be\label{eq:hat Q epsilon n}
 \wh{Q}(\epsilon_n)\wh{\Sigma}_{XX} - Q(\epsilon_n)\Sigma_{XX}
= \frac{1}{\pi}\int_0^\infty t^{-1/2} \left[ \wh{Q}^2(\epsilon_n + t)\wh{\Sigma}_{XX} - Q^2(\epsilon_n + t)\Sigma_{XX} \right] dt.
\ee
Since $\wh{Q}^2(\cdot)$ and $\wh{\Sigma}_{XX}$ commute, and $Q^2(\cdot)$ and $\Sigma_{XX}$ commute, we have
\be\label{eq:hat Q 2 u}
\wh{Q}^2(u)\wh{\Sigma}_{XX} - Q^2(u)\Sigma_{XX} 
&=& \wh{Q}^2(u) [\wh{\Sigma}_{XX} Q^{-2}(u) - \wh{Q}^{-2}(u) \Sigma_{XX}] Q^2(u) \n\\
&=& u \wh{Q}^2(u) (\wh{\Sigma}_{XX} - \Sigma_{XX}) Q^2(u). 
\ee
Therefore, by \eqref{eq:hat R reg '}, \eqref{eq:hat Q epsilon n}, and \eqref{eq:hat Q 2 u},  
\bse
&& \|\wh{R}\reg' - R_n'\|\op
= \|[\wh{Q}(\epsilon_n)\wh{\Sigma}_{XX} - Q(\epsilon_n)\Sigma_{XX}] R_{XY} \|\op \\
&&\le \frac{1}{\pi}\int_0^\infty t^{-1/2} \|(\epsilon_n+t) \wh{Q}^2(\epsilon_n+t) \|\op \|\wh{\Sigma}_{XX} - \Sigma_{XX}\|\op \|Q^2(\epsilon_n+t) R_{XY}\|\op dt.
\ese
Note that $\|(\epsilon_n+t) \wh{Q}^2(\epsilon_n+t) \|\op \le \|I\|\op = 1$, and $\|\wh{\Sigma}_{XX} - \Sigma_{XX}\|\op = O_p(n^{-1/2})$. Thus,
 \bse
\|[\wh{Q}(\epsilon_n)\wh{\Sigma}_{XX} - Q(\epsilon_n)\Sigma_{XX}] R_{XY} \|\op 
&\!\!\!\le&\!\!\! \frac{1}{\pi}  \|\wh{\Sigma}_{XX} - \Sigma_{XX}\|\op \int_0^\infty t^{-1/2}  \|Q^2(\epsilon_n+t) R_{XY}\|\op dt \\
&\!\!\!=&\!\!\! O_p ( n^{-1/2} ) \int_0^\infty t^{-1/2}  \|Q^2(\epsilon_n+t) R_{XY}\|\op dt. 
\ese
By Assumption~\ref{ass:beta}, $R_{XY} = \Sigma_{XX}^\dagger \Sigma_{XX}^{1+\beta}S_{XY} = \Sigma_{XX}^\beta S_{XY}$.
 We consider the following two cases.
\begin{enumerate}
\item  If $0<\beta<1/2$, then
\bse
\|Q^2(\epsilon_n+t) R_{XY}\|\op 
&=& \|Q^2(\epsilon_n+t) \Sigma_{XX}^\beta S_{XY}\|\op \\
&\le& \|(\Sigma_{XX}+(\epsilon_n+t)I)^{-1} \Sigma_{XX}^\beta\|\op \|S_{XY}\|\op \\
&\le& \|(\Sigma_{XX}+(\epsilon_n+t)I)^{\beta-1} \|\op \|S_{XY}\|\op \\
&\le& (\epsilon_n+t)^{\beta-1} \|S_{XY}\|\op.
\ese
  Thus,
\bse
\int_0^\infty t^{-1/2}  \|Q^2(\epsilon_n+t) R_{XY}\|\op dt
\le \|S_{XY}\|\op \int_0^\infty t^{-1/2}  (\epsilon_n+t)^{\beta-1} dt.
\ese
We consider the integral in the right-hand side in detail. On $[0,\epsilon_n]$, we have
\bse
\int_0^{\epsilon_n} t^{-1/2}  (\epsilon_n+t)^{\beta-1} dt \le \epsilon_n^{\beta-1} \int_0^{\epsilon_n} t^{-1/2} dt = \epsilon_n^{\beta-1} 2 t^{1/2} \Big|_0^{\epsilon_n} = 2 \epsilon_n^{\beta-1} \epsilon_n^{1/2} = 2\epsilon_n^{\beta-1/2}.
\ese
On $(\epsilon_n,\infty)$, we have
\bse
\int_{\epsilon_n}^\infty t^{-1/2}  (\epsilon_n+t)^{\beta-1} dt \le  \int_{\epsilon_n}^\infty t^{\beta-3/2}dt = (\beta-1/2)^{-1} t^{\beta-1/2}\Big|_{\epsilon_n}^\infty = (1/2-\beta)^{-1} \epsilon_n^{\beta-1/2}.
\ese
Thus, when $0<\beta<1/2$, we have 
\bse
\int_0^\infty t^{-1/2}  (\epsilon_n+t)^{\beta-1} dt =  O(\epsilon_n^{\beta-1/2}),
\ese
 which implies that
\bse
\|Q^2(\epsilon_n+t) R_{XY}\|\op = O(\epsilon_n^{\beta-1/2}).
\ese

\item If $\beta\ge 1/2$, take any $0<\gamma<1/2$, and then
 \bse
\|Q^2(\epsilon_n+t) R_{XY}\|\op 
&=& \|Q^2(\epsilon_n+t) \Sigma_{XX}^\beta S_{XY}\|\op \\
&\le& \|(\Sigma_{XX}+(\epsilon_n+t)I)^{-1} \Sigma_{XX}^{\gamma}\|\op \|\Sigma_{XX}^{\beta-\gamma}\|\op \|S_{XY}\|\op \\
&\le& \|(\Sigma_{XX}+(\epsilon_n+t)I)^{\gamma-1} \|\op \|\Sigma_{XX}^{\beta-\gamma}\|\op\|S_{XY}\|\op \\
&\le& (\epsilon_n+t)^{\gamma-1} \|\Sigma_{XX}^{\beta-\gamma}\|\op \|S_{XY}\|\op.
\ese
Thus,
\bse
\int_0^\infty t^{-1/2}  \|Q^2(\epsilon_n+t) R_{XY}\|\op dt
\le \|\Sigma_{XX}^{\beta-\gamma}\|\op \|S_{XY}\|\op \int_0^\infty t^{-1/2}  (\epsilon_n+t)^{\gamma-1} dt.
\ese
Using the same argument as in the case of $0<\beta<1/2$, where $\beta$ is replaced by $\gamma$, we have
\bse
\|Q^2(\epsilon_n+t) R_{XY}\|\op = O (\epsilon_n^{\gamma-1/2})
\ese
for any $0<\gamma<1/2$. 

Note that  the above rate is sub-polynomial in the sense that, if we take $\gamma$ to be arbitrarily close to $1/2$, it is slower than any pre-assigned $n^{-c}$ rate.    Nevertheless,  we cannot conclude an  $O(1)$ rate for this term. 
\end{enumerate}
Summarizing the two cases above, we have
\be\label{eq:rreg-rn-rate}
\|\wh{R}\reg' - R_n'\|\op = 
\begin{cases}
    O_p(n^{-1/2} \epsilon_n^{\beta-1/2}), & \text{ if }\beta<1/2,\\
    O_p(n^{-1/2} \epsilon_n^{\gamma-1/2}), & \text{ for any } 0<\gamma<1/2, \text{ if }\beta\ge 1/2.
\end{cases}
\ee

\paragraph{Convergence rate for $R_n'-R_{XY}'$.}
Since,  by Assumption~\ref{ass:beta},   $\Sigma_{XY}=\Sigma_{XX}^{1+\beta}S_{XY}$, we have
\bse
R_n'-R_{XY}'
=[Q(\epsilon_n)\Sigma_{XX}^{1+\beta}-\Sigma_{XX}^{1/2+\beta}]S_{XY}.
\ese
Hence,
\be\label{eq:rn-r0}
\|R_n'-R'_{XY}\|\op
\le \|S_{XY}\|\op
\|Q(\epsilon_n)\Sigma_{XX}^{1+\beta}-\Sigma_{XX}^{1/2+\beta}\|\op.
\ee
We now bound the term $\|Q(\epsilon_n)\Sigma_{XX}^{1+\beta}-\Sigma_{XX}^{1/2+\beta}\|\op$.  Define the functions
\be
\psi_{\epsilon_n}(\lambda) = (\lambda+\epsilon_n)^{-1/2}\lambda^{1+\beta} - \lambda^{1/2 + \beta}, \ \lambda \ge 0, \quad   \psi_{\epsilon_n}(\Sigma_{XX}) = \sum_{j=1}^\infty \psi_{\epsilon_n}(\lambda_j) (\varphi_j \otimes \varphi_j). 
\ee 
Note that the argument of $\psi_{\epsilon_n} (\lambda)$ is a scalar and the argument of $\psi_{\epsilon_n} (\Sigma_{XX})$ is a linear operator. 
By construction, $\psi_{\epsilon_n}(\Sigma_{XX}) = Q(\epsilon_n) \Sigma_{XX}^{1+\beta} - \Sigma_{XX}^{1/2 + \beta} $. 
Hence, by continuous functional calculus, (see, for example, \cite{Conway1990} Chapter~II item 7.11 (b) or \cite{reed1980methods} Theorem~VII.1), we have 
\bse
\|Q(\epsilon_n)\Sigma_{XX}^{1+\beta}-\Sigma_{XX}^{1/2+\beta}\|\op
&=&\sup_{j\ge 1}
\left|(\lambda_j+\epsilon_n)^{-1/2}\lambda_j^{1+\beta}-\lambda_j^{1/2+\beta}\right|\n\\
&\le&\sup_{0<\lambda\le\lambda_1}
\left|(\lambda+\epsilon_n)^{-1/2}\lambda^{1+\beta}-\lambda^{1/2+\beta}\right| \n \\
&=& \sup_{0 < \lambda \leq \lambda_1} \lambda^{1/2+\beta}\left(1- \sqrt{\frac{\lambda}{\lambda+\epsilon_n}}\right)
\ese
To bound the right-hand side above, let
\bse
g_\beta(\lambda)
= \lambda^{1/2+\beta}\left(1- \sqrt{\frac{\lambda}{\lambda+\epsilon_n}}\right), \quad 0 < \lambda \le \lambda_1. 
\ese
Clearly,
\bse
g_\beta(\lambda)= \frac{\lambda^{1/2+\beta}\epsilon_n}{\sqrt{\lambda+\epsilon_n}(\sqrt{\lambda+\epsilon_n}+\sqrt{\lambda})} 
\le \frac{\lambda^{1/2+\beta}\epsilon_n}{\lambda+\epsilon_n}.
\ese
If $0<\lambda<\epsilon_n$, then
\bse
g_\beta(\lambda) \le \frac{\epsilon_n^{1/2+\beta}\epsilon_n}{\epsilon_n} = \epsilon_n^{1/2+\beta}.
\ese
 If $\epsilon_n\le\lambda\le\lambda_1$, then 
\bse
g_\beta(\lambda) \le \frac{\lambda^{1/2+\beta}\epsilon_n}{\lambda} = \epsilon_n \lambda^{-1/2+\beta}.
\ese
To further analyze the order of $g_\beta (\lambda)$  when $\epsilon_n\le\lambda\le\lambda_1$ according to the values of $\beta$, consider the following cases:
\begin{enumerate}
    \item If $\beta<1/2$, then $\lambda^{-1/2+\beta}$ is a decreasing function of $\lambda$, which is maximized at $\lambda=\epsilon_n$. So
    \bse
    g_\beta(\lambda)\le \epsilon_n \epsilon_n^{-1/2+\beta} = \epsilon_n^{1/2+\beta}.
    \ese
    \item If $\beta>1/2$, then $\lambda^{-1/2+\beta}$ is an increasing function of $\lambda$, which is maximized at $\lambda=\lambda_1$. So 
    \bse
    g_\beta(\lambda)\le \epsilon_n \lambda_1^{-1/2+\beta}.
    \ese
    \item If $\beta=1/2$, then $g_\beta(\lambda) \le \epsilon_n$.
\end{enumerate}
 Summarizing the results above, we have the following result:
\begin{enumerate}
    \item If $\beta<1/2$, then 
    \bse
    \sup_{0<\lambda\le\lambda_1} g_\beta(\lambda) \le \epsilon_n^{1/2+\beta}.
    \ese
    \item If $\beta\ge 1/2$, then
    \bse
    \sup_{0<\lambda\le\lambda_1} g_\beta(\lambda) \le \max\{\epsilon_n^{1/2+\beta}, \epsilon_n \lambda_1^{-1/2+\beta}\} = O(\epsilon_n).
    \ese
\end{enumerate}
Combining the above two cases gives
\bse
\sup_{0<\lambda\le\lambda_1} g_\beta(\lambda) = O(\epsilon_n^{(1/2+\beta)\land 1}).
\ese
It follows that $\|Q(\epsilon_n)\Sigma_{XX}^{1+\beta}-\Sigma_{XX}^{1/2+\beta}\|\op = O(\epsilon_n^{(1/2+\beta)\land 1}).$ Substituting this bound into \eqref{eq:rn-r0} and using that $S_{XY}$ is bounded under Assumption~\ref{ass:beta}, we have 
\be\label{eq:bias-rate}
\|R_n'-R_{XY}'\|\op = O(\epsilon_n^{(1/2+\beta)\land 1}).
\ee

\paragraph{Convergence rate for $\wh{R}\res'$.} We further decompose $\wh{R}\res'$ into
\be\label{eq:rres-decomp}
\wh{R}\res' = (\wh{Q}(\epsilon_n)\wh{\Sigma}_{XU} - \wh{Q}(\epsilon_n)\wt{\Sigma}_{XU}) + (\wh{Q}(\epsilon_n)\wt{\Sigma}_{XU} - Q(\epsilon_n)\wt{\Sigma}_{XU}) + Q(\epsilon_n)\wt{\Sigma}_{XU}.
\ee
Since $\|\wh{Q}(\epsilon_n)\|\op \le \epsilon_n^{-1/2}$ and $\|\wh{\Sigma}_{XU} - \wt{\Sigma}_{XU}\|\hs = O_p(n^{-1})$ by Lemma~\ref{lem:sigma-xu-rate}, we have 
\be\label{eq:hat Q epsilon n hat}
\|\wh{Q}(\epsilon_n)\wh{\Sigma}_{XU} - \wh{Q}(\epsilon_n)\wt{\Sigma}_{XU}\|\op = O_p(n^{-1} \epsilon_n^{-1/2}).
\ee
We now find the rate of $\|\wh{Q}(\epsilon_n)\wt{\Sigma}_{XU} - Q(\epsilon_n)\wt{\Sigma}_{XU}\|\op$. By \eqref{eq:q-eps} and \eqref{eq:qhat-eps}, we have
\bse 
\wh{Q}(\epsilon_n) - Q(\epsilon_n)
&=& \frac{1}{\pi}\int_0^\infty t^{-1/2} \left[ \wh{Q}^2(\epsilon_n + t) - Q^2(\epsilon_n + t) \right] dt \n \\
&=& \frac{1}{\pi}\int_0^\infty t^{-1/2} \left[ \wh{Q}^2(\epsilon_n + t) (\Sigma_{XX} - \wh{\Sigma}_{XX}) Q^2(\epsilon_n + t) \right] dt.
\ese
Hence,
\be\label{eq:norm hat Q epsilon n}
\|\wh{Q}(\epsilon_n)\wt{\Sigma}_{XU} - Q(\epsilon_n)\wt{\Sigma}_{XU}\|\op \hspace{3in}\n\\
\le \frac{1}{\pi}\int_0^\infty t^{-1/2} \|\wh{Q}^2(\epsilon_n + t)\|\op \|\Sigma_{XX} - \wh{\Sigma}_{XX}\|\op \|Q^2(\epsilon_n + t)\wt{\Sigma}_{XU}\|\op dt.
\ee
Noticing that  $\|\wh{Q}^2(\epsilon_n + t)\|\op \le (\epsilon_n + t)^{-1}$ and $\|\wh{\Sigma}_{XX} - \Sigma_{XX}\|\op = O_p(n^{-1/2})$ by Lemma~\ref{lem:sigmaxx-hat-rate}, we focus on the rate of $\|Q^2(\epsilon_n + t)\wt{\Sigma}_{XU}\|\op$.

Since, for any $u > 0$, each eigenvalue of $Q^2 (u)$ is a nonincreasing function of $u$, we have, 
for any $\psi \in \ca H_X$, $\| Q^2 (\epsilon_n + t) \psi \|_{\ca H_X} \le \| Q^2 ( \epsilon_n ) \psi \|_{\ca H_X}$. In particular, for any $\theta \in \ca H_Y$,  
\bse
  \| Q^2 (\epsilon_n + t) \wt \Sigma_{XU} \theta  \|_{\ca H_X} \le \| Q^2 ( \epsilon_n ) \wt \Sigma_{XU} \theta \|_{\ca H_X}.  
\ese
Hence $\| Q^2 ( \epsilon_n + t) \wt \Sigma_{XU} \| \op \le \| Q^2 ( \epsilon_n) \wt \Sigma_{XU} \| \op$.
By \eqref{eq:res-3}, noticing that $Q^2 (\epsilon_n) = V_n$, we have 
\be\label{eq:qsq-sigma-xu-rate}
 \|Q^2(\epsilon_n)\wt{\Sigma}_{XU}\|\op   =O_p(n^{-1/2}\epsilon_n^{-(\alpha+1)/(2\alpha)}). 
\ee

Therefore, the right-hand side of (\ref{eq:norm hat Q epsilon n}) is no more than 
\be\label{eq:norm Sigma XX}
&& \|\Sigma_{XX} - \wh{\Sigma}_{XX}\|\op \, \|Q^2 ( \epsilon_n ) \wt{\Sigma}_{XU}\|\op \, \, \frac{1}{\pi}\int_0^\infty  t^{-1/2} (\epsilon_n + t)^{-1}  dt \n\\
&& = O_p ( n^{-1/2}) \, O_p(n^{-1/2}\epsilon_n^{-(\alpha+1)/(2\alpha)}) \, \int_0^\infty  t^{-1/2} (\epsilon_n + t)^{-1}  dt. \hspace{.1in}
\ee
Since
\be\label{eq:int 0 infty t}
\int_0^\infty t^{-1/2} (\epsilon_n + t)^{-1}dt = 2\int_0^\infty (\epsilon_n + s^2)^{-1}ds = 2 \epsilon_n^{-1/2} \arctan(s)\Big|_0^\infty = \pi \epsilon_n^{-1/2}, 
\ee
we have  
\be\label{eq:Q epsilon n tilde}
\begin{split}
\|\wh{Q}(\epsilon_n)\wt{\Sigma}_{XU} - Q(\epsilon_n)\wt{\Sigma}_{XU}\|\op
&=&O_p(n^{-1/2})O_p(n^{-1/2}\epsilon_n^{-(\alpha+1)/(2\alpha)}) \epsilon_n^{-1/2} \\
&=& O_p(n^{-1}\epsilon_n^{-1-1/(2\alpha)}). \hspace{1.2in}
\end{split}
\ee

To derive the convergence rate of the last term in \eqref{eq:rres-decomp}, 
note that 
\bse
 Q(\epsilon_n)\wt{\Sigma}_{XU}  = n^{-1}\sumi Q(\epsilon_n)[\{\ka(\cdot, X_i ) - \mu_X\} \otimes U_i].
\ese
Note that $(X_1,U_1), \dots, (X_n, U_n)$ are i.i.d., 
by the same arguments as \eqref{eq:e-vn-sigma-xu-hs2} and \eqref{eq:vku-hs2} with $V_n$ replaced by $Q(\epsilon_n)$, we have 
\bse
&&E(\|Q(\epsilon_n) \wt{\Sigma}_{XU}\|_\HS^2) 
= n^{-1}E (\|Q(\epsilon_n)[\{\ka(\cdot, X) - \mu_X\} \otimes U]\|_\HS^2), \\
&&\|Q(\epsilon_n)[\{\ka(\cdot, X) - \mu_X\} \otimes U]\|_\HS^2 \n 
= \|U\|^2_{\mH_Y} \sum_{j=1}^\infty (\lambda_j + \epsilon_n)^{-1}\zeta_j^2.
\ese
So 
\bse
E(\|Q(\epsilon_n) \wt{\Sigma}_{XU}\|_\HS^2) = n^{-1} E\left\{\|U\|^2_{\mH_Y} \sum_{j=1}^\infty (\lambda_j + \epsilon_n)^{-1}\zeta_j^2\right\}.
\ese
Applying a similar argument to \eqref{eq:e-vku-hs2} with $V_n$ replaced by $Q(\epsilon_n)$, 
noticing that \eqref{eq:u:bd} still holds under Assumption~\ref{ass:k bdd}, 
we have
\bse
E\left\{\|U\|^2_{\mH_Y} \sum_{j=1}^\infty (\lambda_j + \epsilon_n)^{-1}\zeta_j^2\right\} 
\le 4 C E\left\{ \sum_{j=1}^\infty (\lambda_j + \epsilon_n)^{-1}\zeta_j^2\right\} 
= 4C \sum_{j=1}^\infty \frac{\lambda_j}{\lambda_j + \epsilon_n}
= O(\epsilon_n^{-1/\alpha}),
\ese
 where the last line follows from Lemma~\ref{lem:alpha}. Thus we have shown that 
\bse
E\|Q(\epsilon_n) \wt{\Sigma}_{XU}\|_{\mathrm{HS}}^2=O(n^{-1}\epsilon_n^{-1/\alpha}).
\ese
 By Markov's inequality, for any $K>0$,
\bse
P(\| Q(\epsilon_n) \wt{\Sigma}_{XU}\|_{\HS} > n^{-1/2}\epsilon_n^{-1/(2\alpha)} K) 
&\le& (n^{-1/2}\epsilon_n^{-1/(2\alpha)}K)^{-2} E \| Q(\epsilon_n) \wt{\Sigma}_{XU}\|_{\HS}^2 \\
&=& n \epsilon_n^{1/\alpha} K^{-2} E \| Q(\epsilon_n) \wt{\Sigma}_{XU}\|_{\HS}^2 = O(K^{-2}),
\ese
and consequently, 
\be\label{eq:norm Q epsilon n tilde}
 \|Q(\epsilon_n) \wt{\Sigma}_{XU}\|_{\OP} \le  \| Q(\epsilon_n) \wt{\Sigma}_{XU}\|_{\HS} = O_p(n^{-1/2}\epsilon_n^{-1/(2\alpha)}) . 
\ee

Returning now to \eqref{eq:rres-decomp}, by the triangular inequality, 
\bse 
\|\wh{R}\res'\|\op 
\le \|\wh{Q}(\epsilon_n)\wh{\Sigma}_{XU} - \wh{Q}(\epsilon_n)\wt{\Sigma}_{XU}\|\op + \|\wh{Q}(\epsilon_n)\wt{\Sigma}_{XU} - Q(\epsilon_n)\wt{\Sigma}_{XU}\|\op + \|Q(\epsilon_n)\wt{\Sigma}_{XU}\|\op.
\ese
Substitute \eqref{eq:hat Q epsilon n hat},  \eqref{eq:Q epsilon n tilde}, and (\ref{eq:norm Q epsilon n tilde}) into the right and side  to obtain 
\be\label{eq:rres-rate}
\|\wh{R}\res'\|\op 
&=& O_p(n^{-1} \epsilon_n^{-1/2}) + O_p(n^{-1}\epsilon_n^{-1-1/(2\alpha)}) + O_p(n^{-1/2}\epsilon_n^{-1/(2\alpha)}) \n\\
&=& O_p(n^{-1}\epsilon_n^{-1-1/(2\alpha)} + n^{-1/2}\epsilon_n^{-1/(2\alpha)}).
\ee

\paragraph{Convergence rate for $\| \wh R_{XY}' - R_{XY} ' \|_{\mathrm{OP}}$.} Combining \eqref{eq:R-decomposition}, \eqref{eq:rreg-rn-rate}, \eqref{eq:bias-rate} and \eqref{eq:rres-rate}, we have the following results:
\begin{enumerate}
\item If $0<\beta<1/2$, then
\bse
\|\wh{R}_{XY}' - R_{XY}'\|\op = O_p(n^{-1/2} \epsilon_n^{\beta-1/2} + \epsilon_n^{\beta+1/2} + n^{-1}\epsilon_n^{-1-1/(2\alpha)} + n^{-1/2}\epsilon_n^{-1/(2\alpha)}).
\ese
\item If $\beta \ge 1/2$, then for any $0<\gamma<1/2$,
\be\label{eq:r-rate-b2} 
\|\wh{R}_{XY}' - R_{XY}'\|\op = O_p(n^{-1/2} \epsilon_n^{\gamma-1/2} + \epsilon_n + n^{-1}\epsilon_n^{-1-1/(2\alpha)} + n^{-1/2}\epsilon_n^{-1/(2\alpha)}).
\ee

Since \eqref{eq:r-rate-b2} holds for any $0<\gamma<1/2$ and the first term $n^{-1/2}\epsilon_n^{\gamma -1/2}$ decreases as $\gamma$ increases, 
 if we take $1/2 - 1/(2\alpha) < \gamma < 1/2$, then the first term will be dominated by the fourth term $n^{-1/2}\epsilon_n^{-1/(2\alpha)}$. Therefore, \eqref{eq:r-rate-b2} can be simplified as
\bse
\|\wh{R}_{XY}' - R_{XY}'\|\op = O_p(\epsilon_n + n^{-1}\epsilon_n^{-1-1/(2\alpha)} + n^{-1/2}\epsilon_n^{-1/(2\alpha)}).
\ese
\end{enumerate}
Combining the above two cases gives the desired rate \eqref{eq:r-rate-2}.
\end{proof}

\begin{proof}[Proof of Corollary \ref{cor:f-rate-2}]
By the definitions of $\psi _j$ and $\wh{\psi}_j$, we have
\bse
\psi_j &=& \mu_j^{-1} M' \psi_j = \mu_j^{-1} \Sigma_{XX}^{\dagger 1/2} \Sigma_{XY} \Sigma_{YX} \Sigma_{XX}^{\dagger 1/2} \psi_j, \\
\wh{\psi}_j &=& \wh{\mu}_j^{-1} \wh{M}' \wh{\psi}_j = \wh{\mu}_j^{-1} \wh{Q}(\epsilon_n) \wh{\Sigma}_{XY} \wh{\Sigma}_{YX} \wh{Q}(\epsilon_n) \wh{\psi}_j.
\ese
Therefore,  
\bse
\eta_j = \Sigma_{XX}^{\dagger 1/2} \psi_j
= \mu_j^{-1} \Sigma_{XX}^{\dagger} \Sigma_{XY} \Sigma_{YX} \Sigma_{XX}^{\dagger 1/2} \psi_j 
= \mu_j^{-1} R_{XY} R_{XY}'^* \psi_j.
\ese
and
\bse
\wh \eta_j = \wh{Q}(\epsilon_n)\wh{\psi}_j 
= \wh{\mu}_j^{-1} \wh{Q}^2(\epsilon_n) \wh{\Sigma}_{XY} \wh{\Sigma}_{YX} \wh{Q}(\epsilon_n) \wh{\psi}_j 
= \wh{\mu}_j^{-1}  \wh{R}_{XY} \wh{R}_{XY}'^*  \wh{\psi}_j,
\ese
Consequently, 
\be\label{eq:f-rate2-decomp}
\wh \eta_j  - s_j' \eta_j  
&=& \wh{\mu}_j^{-1}  \wh{R}_{XY} \wh{R}_{XY}'^*  \wh{\psi}_j -  s_j'\mu_j^{-1} R_{XY} R_{XY}'^* \psi_j \n\\
&=&  \wh \mu_j^{-1} \wh R_{XY}\wh R'^*_{XY} (\wh \psi_j - s_j' \psi_j) + s_j' \wh \mu_j^{-1} \wh R_{XY}(\wh R_{XY}'^* - R_{XY}'^*) \psi_j \n \\
&& + s_j' \wh \mu_j^{-1} (\wh R_{XY} - R_{XY})R_{XY}'^* \psi_j +  s_j'(\wh \mu_j^{-1} - \mu_j^{-1}) R_{XY}R_{XY}'^* \psi_j.
\ee
Using a modified version of Lidskii’s inequality (see, for example, Section 2.2 of \cite{koltchinskii2017normal}), we have
\bse
| \wh{\mu}_j - \mu_j | \le \|\wh{M}'-M'\|\op = O_p(r_n').
\ese
Since $\mu_j \ge \mu_d > 0$ and $r_n' = o(1)$, we have $P(\wh{\mu}_j > \mu_j/2) \to 1$. So with probability tending to 1,  we have 
$
| \wh{\mu}_j^{-1} - \mu_j^{-1} | \le 2 \mu_j^{-2} |\wh{\mu}_j - \mu_j|, 
$
which implies
$
| \wh{\mu}_j^{-1} - \mu_j^{-1} | = O_p(r_n').
$
  Moreover, by Theorem~\ref{thm:r-rate}, Theorem~ \ref{thm:r-rate-2}, and Corollary \ref{cor:m-rate-2}, we have
\bse
\|\wh{R}_{XY} - R_{XY}\|\op = O_p(r_n), \quad
\|\wh{R}_{XY}' - R_{XY}'\|\op = O_p(r_n'), \quad
\|\wh{\psi}_j - s_j' \psi_j\|_{\mH_X} = O_p(r_n').
\ese
Since $R_{XY}$ and $R_{XY}'$ are bounded and $r_n,r_n' = o(1)$, taking the $\mH_X$-norm in \eqref{eq:f-rate2-decomp}, we have
\bse
\|\wh \eta_j - s_j' \eta_j \|_{\mH_X} &=&O_p(| \wh{\mu}_j^{-1} - \mu_j^{-1} |) + O_p(\|\wh{R}_{XY} - R_{XY}\|\op) \\
&& + O_p(\|\wh{R}_{XY}' - R_{XY}'\|\op) + O_p(\|\wh{\psi}_j - s_j' \psi_j\|_{\mH_X}) \\
&=& O_p(r_n + r_n')\\
&=& O_p(r_n), 
\ese
as desired.
\end{proof}

\bibliographystyle{agsm}
\bibliography{reference}
\end{document}